\documentclass[11pt]{article}%
\usepackage{amssymb,amsmath,amsfonts,amsthm,array,bm,bbm,color}%
\usepackage{graphicx}
\providecommand{\U}[1]{\protect\rule{.1in}{.1in}}
\numberwithin{equation}{section}

\newtheorem{theorem}{Theorem}[section]
\newtheorem{lemma}[theorem]{Lemma}

\newtheorem{definition}[theorem]{Definition}

\def\<{\langle}
\def\>{\rangle}
\def\d{{\rm d}}
\def\L{\mathcal{L}}
\def\div{{\rm div}}
\def\E{\mathbb{E}}

\def\N{\mathbb{N}}
\def\P{\mathbb{P}}
\def\R{\mathbb{R}}
\def\T{\mathbb{T}}
\def\Z{\mathbb{Z}}

\def\eps{\varepsilon}

\begin{document}

\title{Regularization by transport noise for 3D MHD equations}
\author{Dejun Luo\footnote{Email: luodj@amss.ac.cn. Key Laboratory of RCSDS, Academy of Mathematics and Systems Science, Chinese Academy of Sciences, Beijing, China and School of Mathematical Sciences, University of the Chinese Academy of Sciences, Beijing, China.} }

\maketitle

\begin{abstract}
We consider the problem of regularization by noise for the three dimensional magnetohydrodynamical (3D MHD) equations. It is shown that, in a suitable scaling limit, multiplicative noise of transport type gives rise to bounds on the vorticity fields of the fluid velocity and magnetic fields. As a result, if the noise intensity is big enough, then the stochastic 3D MHD equations admit a pathwise unique global solution for large initial data, with high probability.
\end{abstract}



\section{Introduction}

The topic of regularization by noise for ODEs and PDEs has attracted a lot of attention in the past two decades. There are mainly two different types of noises: additive noise and multiplicative noise. Additive noise has been very successful in improving the well posedness of finite dimensional ODEs, see e.g \cite{Kry-Roeck, FedFla11} (there are many extensions to SDEs with nondegenerate multiplicative noise, see \cite{Veret, Zha05, Zha11} among others); by the method of characteristics, one can solve linear transport equations perturbed by transport noise  \cite{FGP, FedFla13}. Nondegenerate additive noise was also used to regularize abstract stochastic evolution equations on infinite dimensional Hilbert space \cite{DaPFla, DFPR, DFRV}, but the nonlinearity treated therein are not of the advection form in fluid dynamics. The 3D Navier-Stokes equations driven by additive noise have been studied in \cite{FlaRom02, DapDeb, FlaRom08}, yielding some regularity results not available in the deterministic theory.

Concerning multiplicative noise, we focus on the one of transport type modeling perturbations due to some background motion, see \cite{BCF, MikRoz} for some early results. Later on, transport noise was shown to regularize some inviscid models, including the point vortex model of 2D Euler and mSQG equations \cite{FGP11, LuoS}, and point charges of 1D Vlasov-Poisson equations \cite{DFV}. In the recent works \cite{FlaLuoAoP, Galeati, FGL, LuoSaal}, it was shown that linear transport equations and Euler type equations perturbed by transport noises converge, under a suitable scaling of the noises, to the corresponding (deterministic) parabolic equations; moreover, the limiting equations have large viscosity coefficient if the noise intensity is big enough. We have applied this idea to the vorticity form of 3D Navier-Stokes equations \cite{FlaLuo21}, obtaining long-term well posedness for large initial data, with high probability; see \cite{FHLN20} for similar regularization by a deterministic field and \cite{FGL21} for further results. Such phenomenon of dissipation enhancement has some similarity with the theory of stabilization by noise \cite{Arnold, ArnoldCW}, and it has been studied intensively in the deterministic setting, see \cite{Constantin, Zl10, BCZ17, FI19, Iyer, WZZ20} and the references therein. We refer to \cite{Gyongy, ButMyt, GassiatGess, GessMau} for other regularization by noise results and \cite{Gess, BF20} for surveys in the field.

The purpose of this paper is to apply the idea in \cite{FlaLuo21} to the 3D MHD equations on the torus $\T^3= \R^3/\Z^3$:
  \begin{equation}\label{MHD}
  \left\{ \aligned
  & \partial_t u + u\cdot\nabla u - S B\cdot\nabla B + \nabla\big(p + S|B|^2/2 \big)= Re^{-1} \Delta u, \\
  & \partial_t B + u\cdot\nabla B - B\cdot\nabla u = Rm^{-1} \Delta B,\\
  & \nabla\cdot u=0, \quad \nabla\cdot B=0,
  \endaligned \right.
  \end{equation}
which describe the motion of an electrically conductive fluid in a magnetic field. Here, $u=(u_1, u_2, u_3)$ is the fluid velocity field and $B=(B_1, B_2, B_3)$ the magnetic field, $Re$ is the Reynolds number and $Rm$ is the magnetic Reynolds, and $S=M^2/Re Rm$ with $M$ being the Hartman number. Similarly to the theory of 3D Navier-Stokes equations (cf. \cite{DuvLions, SerTem}), for any $L^2$-initial data $(u_0,B_0)$, there exist global weak solutions to \eqref{MHD} (uniqueness is open); while if $(u_0,B_0)$ has $H^1$-regularity, then one can show the local existence of a unique strong solution (its global existence remains open). Global well posedness of the Cauchy problem for a 3D incompressible MHD type system was proved in \cite{LinZha} for smooth initial data which are close enough to the equilibrium state $(x_3,0)$, see also \cite{AZ17} and \cite{WZ18} for related results. There are also lots of studies on the stochastic 2D and/or 3D MHD equations, see \cite{SriSun, BarDaP, Sango, Zeng}. According to the classical theory of 3D MHD equations, the vorticity (in particular, its direction field) plays a key role in the well posedness results, see e.g. \cite{HeXin}.

We are concerned with the vorticity formulation of \eqref{MHD}. To this end, let
  $$\xi= \nabla\times u,\quad \eta= \nabla\times B $$
be the vorticity fields of the velocity and magnetic vector fields, respectively; conversely, $u$ (resp. $B$) can be expressed by $\xi$ (resp. $\eta$) via the Biot-Savart law. As in \cite[(2.1)]{HeXin}, the above system can be rewritten as
  \begin{equation}\label{MHD-vorticity}
  \left\{ \aligned
  & \partial_t \xi + \L_u \xi - S \L_B \eta = Re^{-1} \Delta \xi, \\
  & \partial_t \eta + \L_u \eta - \L_B\xi - 2T(B,u) = Rm^{-1} \Delta \eta,
  \endaligned \right.
  \end{equation}
where $\L_u \xi:= u\cdot\nabla \xi - \xi \cdot\nabla u$ is the Lie derivative, and ($\partial_i= \frac{\partial}{\partial x_i},\, i=1,2,3$)
  $$T(B,u)= \begin{pmatrix}
  \partial_2 B \cdot \partial_3 u - \partial_3 B \cdot \partial_2 u \\
  \partial_3 B \cdot \partial_1 u - \partial_1 B \cdot \partial_3 u \\
  \partial_1 B \cdot \partial_2 u - \partial_2 B \cdot \partial_1 u
  \end{pmatrix}. $$
The exact values of $Re$, $Rm$ and $S$ are not important in our analysis below, and thus we assume they are all equal to 1 for simplicity. To write the system \eqref{MHD-vorticity} in a more compact form, we introduce the notation $\Phi= (\xi, \eta)^\ast$ and denote the nonlinear part as
  $$b(\Phi, \Phi)= \begin{pmatrix} b_1(\Phi, \Phi) \\ b_2(\Phi, \Phi) \end{pmatrix}= \begin{pmatrix} \L_u \xi - \L_B \eta \\ \L_u \eta - \L_B\xi - 2T(B,u) \end{pmatrix}; $$
then, the system \eqref{MHD-vorticity} can be simply written as
  \begin{equation}\label{MHD-vorticity-1}
  \partial_t \Phi + b(\Phi, \Phi) = \Delta \Phi.
  \end{equation}

We introduce a few notations of functional spaces. For $s\in \R$, let $H^s(\T^3, \R^3)$ be the usual Sobolev space of vector fields on $\T^3$, endowed with the norm $\|\cdot \|_{H^s} $; $H^0(\T^3, \R^3)$ coincides with $L^2(\T^3, \R^3)$. For simplicity, we assume in this paper that the vector fields have zero mean, and define the spaces
  $$\aligned
  H &= \big\{\Phi=(\xi, \eta)^\ast: \xi, \eta \in L^2(\T^3, \R^3), \nabla\cdot \xi = \nabla\cdot \eta = 0\big\}, \\
  V &= \big\{\Phi=(\xi, \eta)^\ast \in H: \xi, \eta \in H^1(\T^3, \R^3)\big\}.
  \endaligned $$
We write $\<\cdot, \cdot\>$ or $\<\cdot, \cdot\>_{L^2}$ for the inner product in $H$, with the corresponding norm $\|\cdot \|_H= \|\cdot \|_{L^2}$; $\|\cdot \|_V= \|\cdot \|_{H^1}$ is the norm in $V$. Note that in the periodic case, we do not have to distinguish the spaces for the velocity and the magnetic fields. It is well known that the system \eqref{MHD-vorticity-1} admits a unique local solution for general initial data $\Phi_0=(\xi_0, \eta_0)\in H$; moreover, there exists a small $r_0>0$ such that if $\|\Phi_0 \|_H \leq r_0$, then the unique solution exists globally in time.

Inspired by \cite{FlaLuo21} (see in particular Section 1.2 therein), we perturb the vorticity equations \eqref{MHD-vorticity-1} by transport noise:
  \begin{equation}\label{stoch-MHD}
  \partial_t \Phi + b(\Phi, \Phi) = \Delta \Phi + \Pi\big(\dot W \cdot\nabla \Phi\big), \\
  \end{equation}
where $\Pi$ is the Leray projection operator and $\dot W $ denotes the formal time derivative of a spatially divergence free noise $W(t,x)$ on $\T^3$, and
  \begin{equation}\label{noises}
  \Pi\big(\dot W \cdot\nabla \Phi\big) = \begin{pmatrix} \Pi\big(\dot W \cdot\nabla \xi\big) \\ \Pi\big(\dot W \cdot\nabla \eta\big) \end{pmatrix} .
  \end{equation}
We apply Leray's projection $\Pi$ to the noise terms to make them divergence free; otherwise, the equations are in general not meaningful since the other quantities are all divergence free. We remark that, since the noise $W$ is spatially divergence free, the stochastic 3D MHD equations \eqref{stoch-MHD} have the same energy estimate as the deterministic system \eqref{MHD-vorticity-1}; this implies that the stochastic equations \eqref{stoch-MHD} have unique local solutions for general initial data, while the solutions are global for initial data satisfying $\|\Phi_0 \|_H \leq r_0$, where $r_0$ is the same parameter as in the previous paragraph. It seems that, at first glance, transport noise has no regularizing effect on the 3D MHD system; however, following \cite{FlaLuo21}, we shall show that the noise enhances dissipation in a suitable scaling limit, and thus leads to long-term well posedness for big initial data, with large probability.

We use the same noise as in \cite{FlaLuo21} to perturb the equations:
  $$W(t,x)= \frac{C_\nu}{\|\theta \|_{\ell^2}} \sum_{k\in \Z^3_0} \sum_{\alpha=1}^2 \theta_k \sigma_{k,\alpha}(x) W^{k,\alpha}_t.$$
Here, for some $\nu>0$, $C_\nu= \sqrt{3\nu/2}$ denotes the intensity of the noise; $\Z^3_0$ is the collection of nonzero lattice points and $\theta\in \ell^2 := \ell^2(\Z^3_0)$, the latter being the space of square summable real sequences indexed by $\Z^3_0$. Next, $\{\sigma_{k,\alpha} \}_{ k\in \Z^3_0, \alpha=1,2}$ are divergence free vector fields on $\T^3$ and $\{W^{k,\alpha} \}_{k\in \Z^3_0, \alpha= 1,2}$ are independent standard complex Brownian motions; see Section \ref{subsec-notations} for their precise definitions. Now, the stochastic equations studied in this paper have the following precise form:
  \begin{equation}\label{stoch-MHD.1}
  \d \Phi + b(\Phi, \Phi)\,\d t = \Delta \Phi\,\d t + \frac{C_\nu}{\|\theta \|_{\ell^2}} \sum_{k, \alpha} \theta_k \Pi(\sigma_{k, \alpha} \cdot\nabla \Phi) \circ \d W^{k, \alpha},
  \end{equation}
where $\sum_{k, \alpha}$ stands for $\sum_{k\in \Z^3_0} \sum_{\alpha=1}^2$, $\Pi(\sigma_k \cdot\nabla \Phi)$ is understood as in \eqref{noises} and $\circ$ means we are making use of the Stratonovich stochastic differentiation. We shall always consider those $\theta\in \ell^2$ having only finitely many nonzero components, and assume that
  \begin{equation}\label{theta-sym}
  \theta_k =\theta_l \quad \mbox{for all } |k|=|l|.
  \end{equation}
We can rewrite \eqref{stoch-MHD.1} in the It\^o form as
  \begin{equation}\label{stoch-MHD.2}
  \d \Phi + b(\Phi, \Phi)\,\d t = \big[\Delta \Phi + S_\theta(\Phi) \big]\,\d t + \frac{C_\nu}{\|\theta \|_{\ell^2}} \sum_{k, \alpha} \theta_k \Pi(\sigma_{k,\alpha} \cdot\nabla \Phi)\, \d W^{k,\alpha} ,
  \end{equation}
where the Stratonovich-It\^o corrector is given by
  $$ S_\theta(\Phi)= \frac{C_\nu^2}{\|\theta \|_{\ell^2}^2} \sum_{k,\alpha} \theta_k^2\, \Pi\big[\sigma_{k,\alpha}\cdot\nabla \Pi(\sigma_{-k, \alpha} \cdot\nabla \Phi) \big]. $$
It is a symmetric second order differential operator on vector fields.

Given $\Phi_0 \in H$, we denote by $\Phi(t;\Phi_0, \nu, \theta)$ the unique local solution to \eqref{stoch-MHD.2} with the maximal time of existence $\tau= \tau(\Phi_0, \nu, \theta)$. Recall that $\nu>0$ measures the noise intensity. We write $B_H(K)$ for the ball in $H$ centered at the origin with radius $K>0$. Here is the main result of our paper.

\begin{theorem}\label{thm-main}
Given $K>0$ and small $\eps>0$, there exist big $\nu>0$ and $\theta\in \ell^2$ such that
  $$\P(\tau(\Phi_0, \nu, \theta) =+\infty )> 1-\eps\quad \mbox{for all } \Phi_0\in B_H(K). $$
Equivalently, with large probability uniformly over $\Phi_0\in B_H(K)$, the unique solution to \eqref{stoch-MHD.2} with initial data $\Phi_0$ exists globally in time.
\end{theorem}

For the sake of readers' understanding, we briefly describe here the ideas of proof. Similarly to \cite{FlaLuo21}, we shall take a sequence $\{\theta^N \}_{N\geq 1} \subset \ell^2$ satisfying
  \begin{equation}\label{theta-N}
  \theta^N_k= \frac{{\bf 1}_{\{N\leq |k|\leq 2N\} }}{|k|^\kappa}, \quad k\in \Z^3_0,
  \end{equation}
where $\kappa>0$ is any fixed parameter. It is obvious that
  \begin{equation}\label{theta-N.1}
  \lim_{N\to \infty} \frac{\|\theta^N \|_{\ell^\infty}}{\|\theta^N \|_{\ell^2}} =0;
  \end{equation}
more importantly, for any smooth divergence free vector field $v$ on $\T^3$, it was proved in \cite[Section 5]{FlaLuo21} that
  \begin{equation}\label{theta-N.2}
  \lim_{N\to \infty} S_{\theta^N}(v) = \frac35 \nu \Delta v \quad \mbox{holds in } L^2(\T^3,\R^3).
  \end{equation}
Consider the following sequence of stochastic 3D MHD equations
  \begin{equation}\label{stoch-MHD-N}
  \d \Phi_N + b(\Phi_N, \Phi_N)\,\d t = \big[\Delta \Phi_N + S_{\theta^N}(\Phi_N) \big]\,\d t + \frac{C_\nu}{\|\theta^N \|_{\ell^2}} \sum_{k,\alpha} \theta^N_k \Pi(\sigma_{k,\alpha} \cdot\nabla \Phi_N)\, \d W^{k,\alpha}.
  \end{equation}
In fact, we shall introduce a suitable cut-off in the nonlinear part $b(\Phi_N, \Phi_N)$, see Section 2 for details. Thanks to the limit \eqref{theta-N.1} and $L^2$-bounds on the solutions, we can show that the martingale part in \eqref{stoch-MHD-N} will vanish in the weak sense. Furthermore, the key result \eqref{theta-N.2} implies that the limit equation is
  \begin{equation}\label{eq-limit.1}
  \partial_t \Phi + b(\Phi, \Phi) = \Big(1+ \frac35 \nu\Big) \Delta \Phi.
  \end{equation}
This equation has an enhanced dissipation which comes from the noise intensity $\nu$; hence, \eqref{eq-limit.1} is much better posed than the deterministic system \eqref{MHD-vorticity-1}. Theorem \ref{thm-main} will be proved by using the fact that the solutions of \eqref{stoch-MHD-N} are close to those of \eqref{eq-limit.1}.

The paper is organized as follows. In Section \ref{sec-2}, we first introduce the explicit choices of the vector fields $\{\sigma_{k,\alpha} \}_{ k\in \Z^3_0, \alpha=1,2}$ and the complex Brownian motions  $\{W^{k,\alpha} \}_{ k\in \Z^3_0, \alpha=1,2}$, then we prove the well posedness of stochastic 3D MHD equations with a cut-off. The proof of Theorem \ref{thm-main} will be given in  Section \ref{sec-proof}, where we first prove a scaling limit result (i.e. Theorem \ref{thm-limit}) which is a crucial step for proving the main result. Finally we present in the appendix a sketched proof of the key limit \eqref{theta-N.2}.

\section{Notations and global well posedness of stochastic 3D MHD equations with cut-off} \label{sec-2}

We first introduce in Section \ref{subsec-notations} the definitions of divergence free vector fields $\{\sigma_{k,\alpha} \}_{ k\in \Z^3_0, \alpha=1,2}$ and the complex Brownian motions. In Section \ref{subsec-well-posedness}, we apply the Galerkin approximation and compactness method to prove the existence of weak solutions to stochastic 3D MHD equations \eqref{stoch-MHD-cut} with a cut-off; then, we show that pathwise uniqueness holds for the system, thus, by Yamada-Watanabe type result, we conclude that \eqref{stoch-MHD-cut} admits a unique strong solution.

\subsection{Notations}\label{subsec-notations}

This part is taken from the beginning of \cite[Section 2]{FlaLuo21}. Recall that $\Z^3_0= \Z^3\setminus \{0\}$ is the nonzero lattice points; let $\Z^3_0= \Z^3_+ \cup \Z^3_-$ be a partition of $\Z^3_0$ such that
  $$\Z^3_+ \cap \Z^3_-= \emptyset, \quad \Z^3_+ = - \Z^3_-.$$
Let $L^2_0(\T^3, \mathbb C)$ be the space of complex valued square integrable functions on $\T^3$ with zero average; it has the CONS:
  $$e_k(x)= e^{2\pi {\rm i} k\cdot x}, \quad x\in \T^3,\, k\in \Z^3_0,$$
where ${\rm i}$ is the imaginary unit. For any $k\in \Z^3_+$, let $\{a_{k,1}, a_{k,2}\}$ be an orthonormal basis of $k^\perp := \{x\in \R^3: k\cdot x=0\}$ such that $\{a_{k,1}, a_{k,2}, \frac{k}{|k|}\}$ is right-handed. The choice of $\{a_{k,1}, a_{k,2}\}$ is not unique. For $k\in \Z^3_-$, we define $a_{k,\alpha} = a_{-k,\alpha}$, $\alpha=1,2$. Now we can define the divergence free vector fields:
  \begin{equation}\label{vector-fields}
  \sigma_{k,\alpha}(x) = a_{k,\alpha} e_k(x), \quad x\in \T^3,\,  k\in \Z^3_0,\, \alpha=1,2.
  \end{equation}
Then $\{\sigma_{k,1}, \sigma_{k,2}:  k\in \Z^3_0\}$ is a CONS of the subspace $H_{\mathbb C} \subset L^2_0(\T^3,\mathbb C^3)$ of square integrable and divergence free vector fields with zero mean. A vector field
  $$v= \sum_{k,\alpha} v_{k,\alpha} \sigma_{k,\alpha} \in H_{\mathbb C}$$
has real components if and only if $\overline{v_{k,\alpha}} = v_{-k, \alpha}$.

Next, we introduce the family $\{W^{k,\alpha}: k\in \Z^3_0, \alpha=1,2\}$ of complex Brownian motions. Let
  $$\big\{B^{k,\alpha}: k\in \Z^3_0,\, \alpha=1,2 \big\}$$
be a family of independent standard real Brownian motions; then the complex Brownian motions can be defined as
  $$W^{k,\alpha} = \begin{cases}
  B^{k,\alpha} + {\rm i} B^{-k,\alpha}, & k\in  \Z^3_+;\\
  B^{-k,\alpha} - {\rm i} B^{k,\alpha}, & k\in  \Z^3_-.
  \end{cases}$$
Note that $\overline{W^{k,\alpha}}= W^{-k,\alpha}\, (k\in\Z^3_0, \alpha=1,2)$, and they have the following quadratic covariation:
  \begin{equation}\label{qudratic-var}
  \big[W^{k,\alpha}, W^{l,\beta} \big]_t= 2\, t \delta_{k,-l} \delta_{\alpha, \beta},\quad k,l\in \Z^3_0,\, \alpha, \beta\in \{1,2\} .
  \end{equation}

In the following, for a vector field $X$ in $ H^1(\T^3,\R^3)$, we write $\L_X^\ast$ for the adjoint operator of the Lie derivative $\L_X$: for any $H^1$-vector fields $Y$ and $Z$, $\<\L_X Y, Z\>_{L^2} = -\<Y, \L_X^\ast Z\>_{L^2}$. If $X$ is divergence free, one has $\L_X^\ast Y =X\cdot \nabla Y + (\nabla X)^\ast Y$, where for $i=1,2,3$, $((\nabla X)^\ast Y)_i= Y\cdot \partial_i X$.

\subsection{Global well posedness of \eqref{stoch-MHD-cut}} \label{subsec-well-posedness}

Due to the nonlinear terms in \eqref{stoch-MHD.2}, we can only prove existence of local solutions for general initial data. Therefore, we need a cut-off technique as below. Let $R>0$ be fixed and $f_R\in C_b^1(\R_+, [0,1])$ be a non-increasing function such that $f_R|_{[0,R]} \equiv 1$ and $f_R|_{[R+1,\infty)} \equiv 0$. Consider the stochastic 3D MHD equations with cut-off:
  \begin{equation}\label{stoch-MHD-cut}
  \d \Phi + f_R(\Phi) b(\Phi, \Phi)\,\d t = \big[\Delta \Phi + S_\theta(\Phi) \big]\,\d t + \frac{C_\nu}{\|\theta \|_{\ell^2}} \sum_{k, \alpha} \theta_k \Pi(\sigma_{k,\alpha} \cdot\nabla \Phi)\, \d W^{k,\alpha},
  \end{equation}
where $f_R(\Phi)= f_R\big(\|\Phi \|_{H^{-\delta}} \big)$ for some fixed $\delta\in (0,1/2)$. We remark that, if the $H^{-\delta }$-norm of $\Phi$ does not attain the threshold $R$, then the cut-off function $f_R(\Phi)$ can be dropped and \eqref{stoch-MHD-cut} reduces to \eqref{stoch-MHD.2}. Thanks to the cut-off, we can show that, for any initial data $\Phi_0 \in H$, the above system \eqref{stoch-MHD-cut} admits a pathwise unique global solution, strong in the probabilistic sense and weak in the analytic sense. First of all, we explain what we mean by a solution to \eqref{stoch-MHD-cut}.

\begin{definition}\label{defin}
Given a filtered probability space $(\Omega ,\mathcal{F},(\mathcal{F}_t),\mathbb{P})$ and a family of independent $(\mathcal{F}_t)$-complex Brownian motions $\{W^{k,\alpha}\}_{k\in \Z_0^3, \alpha=1,2}$ defined on $\Omega$, we say that an $(\mathcal{F}_t)$-progressively measurable process $\Phi= (\xi, \eta)^\ast$ is a strong solution to \eqref{stoch-MHD-cut} with initial condition $\Phi_0= (\xi_0, \eta_0)^\ast$ if it has trajectories of class $L^{\infty}(0,T;H) \cap L^{2}( 0,T;V)  $ and in $C\big([0,T],H^{-\delta} \big)$ and, for any divergence free vector field $v\in C^\infty(\mathbb{T}^{3},\mathbb{R}^{3})$, $\mathbb{P}$-a.s. the following identities hold for all $t\in [0,T]$:
  \begin{equation}\label{defin.1}
  \aligned
  \< \xi_{t},v\> =&\, \<\xi_0, v\> + \int_0^t f_R(\Phi_s) \Big[\big\< \xi_s,\L_{u_s}^{\ast}v \big\> - \big\< \eta_s,\L_{B_s}^{\ast}v \big\> \Big] \,\d s\\
  &\, + \int_0^t \< \xi_s,\Delta v + S_\theta(v) \>\,\d s -\frac{C_{\nu}}{\| \theta\|_{\ell^2}} \sum_{k,\alpha}\theta_{k} \int_0^t \<
  \xi_{s}, \sigma_{k,\alpha}\cdot \nabla v \> \,\d W_{s}^{k,\alpha},
  \endaligned
  \end{equation}
  \begin{equation}\label{defin.2}
  \aligned
  \< \eta_t,v\> =&\, \<\eta_0, v\> + \int_0^t f_R(\Phi_s) \Big[\big\< \eta_s,\L_{u_s}^{\ast}v \big\> - \big\< \xi_s,\L_{B_s}^\ast v \big\> -2 \big<T(B,u), v\big> \Big] \,\d s\\
  &\, + \int_0^t \< \eta_s,\Delta v + S_\theta(v) \> \,\d s -\frac{C_{\nu}}{\| \theta\|_{\ell^2}} \sum_{k,\alpha}\theta_{k} \int_0^t \<
  \eta_{s}, \sigma_{k,\alpha}\cdot \nabla v \> \,\d W_{s}^{k,\alpha} .
  \endaligned
  \end{equation}
\end{definition}

Our main result in this part can be stated as follows.

\begin{theorem}\label{thm-existence}
Assume $\Phi_0= (\xi_0, \eta_0)^\ast \in H$, $T>0$ and $\theta\in \ell^2$ verifies the symmetry property \eqref{theta-sym}, then there exists on the interval $[0,T]$ a pathwise unique strong solution $\Phi=(\xi, \eta)^\ast$ to \eqref{stoch-MHD-cut} in the sense of Definition \ref{defin}. Moreover, there is a constant $C_{\|\Phi_0\|_{H}, \delta, R,T}>0$, independent of $\nu>0$ and $\theta\in \ell^2$, such that
  \begin{equation}\label{solution-property}
  \P \mbox{-a.s.}, \quad \|\Phi \|_{L^\infty(0,T; H)} \vee \|\Phi \|_{L^2(0,T; V)} \leq C_{\|\Phi_0\|_H, \delta, R,T}.
  \end{equation}
\end{theorem}

To show the existence of weak solutions to the above equations, we first prove an a priori estimate on the solutions. In the following, we make frequent use of the Sobolev embedding inequality
  \begin{equation}\label{Sobolev-embedding}
  \|\varphi \|_{L^q} \leq c \|\varphi \|_{H^s} \quad \mbox{with} \quad \frac1q = \frac12 - \frac s3,\ s<\frac32.
  \end{equation}
We also need the interpolation inequality: for any $s_0 <s < s_1$,
  $$\|\varphi \|_{H^s} \leq \|\varphi \|_{H^{s_0}}^{(s_1-s)/(s_1-s_0)} \|\varphi \|_{H^{s_1}}^{(s-s_0)/(s_1-s_0)}. $$
Moreover, we write $C$ (without subscript) for generic positive constants independent of the key parameters such as $\delta$ and $R$.

\begin{lemma}\label{lem-apriori}
Let $\Phi_0\in H$. Then, there exists $C_{\delta, R}>0$ such that, $\P$-a.s.,
  $$\|\Phi_t\|_{L^2}^2 + \int_0^t \|\Phi_s\|_{H^1}^2\,\d s \leq \|\Phi_0 \|_{L^2}^2 + C_{\delta, R} T \quad \mbox{for all } t\in [0,T].$$
\end{lemma}

\begin{proof}
We omit the time variable to save notations. In Stratonvich form, \eqref{stoch-MHD-cut} reads as (cf. \eqref{stoch-MHD.1})
  $$\d \Phi + f_R(\Phi) b(\Phi, \Phi)\,\d t = \Delta \Phi\,\d t + \frac{C_\nu}{\|\theta \|_{\ell^2}} \sum_{k, \alpha} \theta_k \Pi(\sigma_{k, \alpha} \cdot\nabla \Phi) \circ \d W^{k, \alpha}.$$
By the Stratonovich calculus,
  $$\d\|\Phi\|_{L^2}^2 = 2\big\<\Phi, -f_R(\Phi) b(\Phi, \Phi)+ \Delta \Phi \big\> \,\d t + \frac{2 C_\nu}{\|\theta \|_{\ell^2}} \sum_{k, \alpha} \theta_k \<\Phi, \Pi(\sigma_{k, \alpha} \cdot\nabla \Phi) \> \circ \d W^{k, \alpha}. $$
Note that the two components of $\Phi=(\xi, \eta)^\ast$ are divergence free; by the integration by parts formula,
  $$\<\Phi, \Pi(\sigma_{k,\alpha} \cdot\nabla \Phi)\> = \<\xi, \Pi(\sigma_{k,\alpha} \cdot\nabla \xi)\> + \<\eta, \Pi(\sigma_{k,\alpha} \cdot\nabla \eta)\> = \<\xi, \sigma_{k,\alpha} \cdot\nabla \xi\> + \<\eta, \sigma_{k,\alpha} \cdot\nabla \eta\> =0, $$
since $\sigma_{k, \alpha}$ is also divergence free. Therefore, the above equation reduces to
  \begin{equation}\label{lem-apriori.1}
  \d\|\Phi \|_{L^2}^2 = -2f_R(\Phi) \<\Phi, b(\Phi, \Phi)\> \,\d t -2\|\nabla \Phi\|_{L^2}^2\,\d t.
  \end{equation}
Recalling the expression of $b(\Phi, \Phi)$, we have
  $$\<\Phi, b(\Phi, \Phi)\> = \<\xi, b_1(\Phi, \Phi)\> + \<\eta, b_2(\Phi, \Phi)\>. $$
We estimate the two quantities separately.

\emph{Step 1.} We have
  \begin{equation}\label{lem-apriori.2}
  \<\xi, b_1(\Phi, \Phi)\> = \<\xi, \L_u\xi \> - \<\xi, \L_B \eta\> =: I_1 + I_2.
  \end{equation}
As $u$ is divergence free, we have $I_1= -\<\xi, \xi\cdot\nabla u \>_{L^2}$ and thus, by H\"older's inequality,
  $$|I_1| \leq \|\xi\|_{L^3}^2 \|\nabla u\|_{L^3} \leq C \|\xi \|_{L^3}^3 \leq C \|\xi \|_{H^{1/2}}^3 .$$
Using the interpolation inequality with $s_0=-\delta$, $s_1= 1$ and $s=1/2$, we obtain
  \begin{equation}\label{lem-apriori.3}
  |I_1| \leq C \|\xi \|_{H^{-\delta}}^{3/2(1+\delta)} \|\xi \|_{H^1}^{3(1+2\delta)/2(1+\delta)} \leq \eps \|\xi \|_{H^1}^2 + C_{\delta, \eps} \|\xi \|_{H^{-\delta}}^{6/(1-2\delta)},
  \end{equation}
where in the last step we have used the inequality $ab\leq \frac{a^p}p + \frac{b^q}q $ for $a,b\geq 0$ and $p= 4(1+\delta)/(6\delta+3),\, q= 4(1+\delta)/(1- 2\delta) $; $\eps>0$ is some fixed constant to be determined later.

Next, we deal with $I_2$:
  $$I_2= - \<\xi, B\cdot\nabla \eta\> + \<\xi,\eta\cdot\nabla B\> = : I_{2,1} +I_{2,2}.$$
By the H\"older inequality with exponents $\frac12 + \frac13 + \frac16 =1$,
  $$|I_{2,1}| \leq \|\nabla \eta\|_{L^2} \|\xi\|_{L^3} \|B\|_{L^6} \leq C \|\eta\|_{H^1} \|\xi\|_{H^{1/2}} \|B\|_{H^1} \leq C \|\eta\|_{H^1} \|\xi\|_{H^{1/2}} \|\eta\|_{L^2}. $$
Using the interpolation inequality with $s=1/2$ and $s=0$, while $s_0=-\delta$ and $s_1= 1$ are fixed, we have
  $$\aligned
  |I_{2,1}| &\leq C \|\eta\|_{H^1} \|\xi \|_{H^{-\delta}}^{1/2(1+\delta)} \|\xi \|_{H^1}^{(1+2\delta)/2(1+\delta)} \|\eta\|_{H^{-\delta}}^{1/(1+\delta)} \|\eta\|_{H^1}^{\delta/(1+\delta)} \\
  &= C\|\xi \|_{H^{-\delta}}^{1/2(1+\delta)} \|\xi \|_{H^1}^{(1+2\delta)/2(1+\delta)} \|\eta\|_{H^{-\delta}}^{1/(1+\delta)} \|\eta\|_{H^1}^{(1+2\delta)/(1+\delta)} .
  \endaligned $$
Now applying the inequality $abcd\leq \frac{a^p}p + \frac{b^q}q + \frac{c^r}r + \frac{d^s}s \ (a,b,c,d\geq 0)$ with
  $$p=\frac{12(1+\delta)}{1-2\delta}, \quad q=\frac{4(1+\delta)}{1+2\delta},\quad r=\frac{6(1+\delta)}{1-2\delta}, \quad s=\frac{2(1+\delta)}{1+2\delta}, $$
we obtain
  \begin{equation}\label{lem-apriori.4}
  |I_{2,1}| \leq  \eps \|\xi \|_{H^1}^2 + \eps \|\eta\|_{H^1}^2 + C_{\delta,\eps} \Big( \|\xi \|_{H^{-\delta}}^{6/(1-2\delta)} + \|\eta \|_{H^{-\delta}}^{6/(1-2\delta)} \Big).
  \end{equation}
In the same way,
  $$\aligned
  |I_{2,2}| &\leq \|\xi \|_{L^3} \|\eta \|_{L^3} \|\nabla B\|_{L^3} \leq C \|\xi \|_{L^3} \|\eta \|_{L^3}^2 \leq C \|\xi \|_{H^{1/2}} \|\eta \|_{H^{1/2}}^2 \\
  &\leq C \|\xi \|_{H^{-\delta}}^{1/2(1+\delta)} \|\xi \|_{H^1}^{(1+2\delta)/2(1+\delta)} \|\eta \|_{H^{-\delta}}^{1/(1+\delta)} \|\eta \|_{H^1}^{(1+2\delta)/(1+\delta)}.
  \endaligned $$
The right hand side can be estimated similarly as for $|I_{2,1}|$, thus we have
  $$ |I_2| \leq |I_{2,1}| + |I_{2,2}| \leq 2\eps \big(\|\xi \|_{H^1}^2 + \|\eta\|_{H^1}^2 \big) + C_{\delta,\eps} \Big( \|\xi \|_{H^{-\delta}}^{6/(1-2\delta)} + \|\eta \|_{H^{-\delta}}^{6/(1-2\delta)} \Big). $$
Combining this result with \eqref{lem-apriori.2} and \eqref{lem-apriori.3} leads to
  \begin{equation}\label{lem-apriori.5}
  |\<\xi, b_1(\Phi, \Phi)\> | \leq 3\eps \big(\|\xi \|_{H^1}^2 + \|\eta\|_{H^1}^2 \big) + C_{\delta,\eps} \Big( \|\xi \|_{H^{-\delta}}^{6/(1-2\delta)} + \|\eta \|_{H^{-\delta}}^{6/(1-2\delta)} \Big).
  \end{equation}

\emph{Step 2}. Now we turn to estimate
  $$\<\eta, b_2(\Phi, \Phi)\> = \<\eta, \L_u \eta\> - \<\eta, \L_B\xi\> - 2\<\eta, T(B,u)\> =: J_1 + J_2 + J_3. $$
The arguments are similar to those in \emph{Step 1}. We have $J_1= - \<\eta, \eta\cdot \nabla u\>_{L^2}$, thus
  $$|J_1| \leq \|\eta \|_{L^3}^2 \|\nabla u \|_{L^3} \leq C \|\eta \|_{L^3}^2 \|\xi \|_{L^3} \leq C \|\xi \|_{H^{1/2}} \|\eta \|_{H^{1/2}}^2.  $$
Repeating the estimate for $|I_{2,2}|$ yields
  \begin{equation*}
  |J_1| \leq \eps \big(\|\xi \|_{H^1}^2 + \|\eta\|_{H^1}^2 \big) + C_{\delta,\eps} \Big( \|\xi \|_{H^{-\delta}}^{6/(1-2\delta)} + \|\eta \|_{H^{-\delta}}^{6/(1-2\delta)} \Big).
  \end{equation*}
Next, the term $J_2$ can be treated in the same way as $I_2$ and we have
  \begin{equation*}
  |J_2|\leq 2\eps \big(\|\xi \|_{H^1}^2 + \|\eta\|_{H^1}^2 \big) + C_{\delta,\eps} \Big( \|\xi \|_{H^{-\delta}}^{6/(1-2\delta)} + \|\eta \|_{H^{-\delta}}^{6/(1-2\delta)} \Big).
  \end{equation*}
Finally, the definition of $T(B,u)$ implies
  $$|J_3| \leq C\|\eta \|_{L^3} \|\nabla B\|_{L^3} \|\nabla u\|_{L^3}\leq C \|\eta \|_{L^3}^2 \|\xi \|_{L^3} \leq C \|\eta \|_{H^{1/2}}^2 \|\xi \|_{H^{1/2}},$$
hence, this last quantity can also be treated as $|I_{2,2}|$. Summarizing these estimates, we obtain
  $$|\<\eta, b_2(\Phi, \Phi)\> | \leq 5\eps \big(\|\xi \|_{H^1}^2 + \|\eta\|_{H^1}^2 \big) + C_{\delta,\eps} \Big( \|\xi \|_{H^{-\delta}}^{6/(1-2\delta)} + \|\eta \|_{H^{-\delta}}^{6/(1-2\delta)} \Big).$$

\emph{Step 3}. Thanks to the estimates in \emph{Steps 1 and 2}, we deduce from \eqref{lem-apriori.1} that
  $$\aligned
  \d\|\Phi \|_{L^2}^2 &\leq 2f_R(\Phi) \Big[8\eps \big(\|\xi \|_{H^1}^2 + \|\eta\|_{H^1}^2 \big) + C_{\delta,\eps} \Big( \|\xi \|_{H^{-\delta}}^{6/(1-2\delta)} + \|\eta \|_{H^{-\delta}}^{6/(1-2\delta)} \Big)\Big]\,\d t -2\|\nabla \Phi\|_{L^2}^2\,\d t \\
  &\leq 16\eps \big(\|\xi \|_{H^1}^2 + \|\eta\|_{H^1}^2 \big) \,\d t + C_{\delta,\eps} (R+1)^{6/(1-2\delta)} \,\d t -2\|\nabla \Phi\|_{L^2}^2\,\d t,
  \endaligned $$
where we have used $0\leq f_R \leq 1$ and $f_R|_{[R+1, \infty)} \equiv 0$. The Poincar\'e inequality implies, for some constant $C_0>0$,
  $$\|\nabla \Phi\|_{L^2}^2 \geq C_0 \|\Phi\|_{H^1}^2= C_0 \big(\|\xi \|_{H^1}^2 + \|\eta\|_{H^1}^2 \big). $$
Therefore, taking $\eps= C_0/16$ leads to
  $$\d\|\Phi \|_{L^2}^2 + C_0\| \Phi\|_{H^1}^2\,\d t \leq C_{\delta,\eps} (R+1)^{6/(1-2\delta)} \,\d t.$$
This gives us the desired estimate.
\end{proof}

With the above a priori estimate in hand, it is standard to apply the Galerkin approximation to show the existence of weak solutions to \eqref{stoch-MHD-cut}, see for instance \cite[Section 3]{FlaLuo21}. Here we sketch the main steps. Let $H_N$ be the finite dimensional subspace of $H$ spanned by the fields $\{\sigma_{k,\alpha}: |k|\leq N, \alpha=1,2\}$. Denote by $\Pi_N : H \to H_N$ be the orthogonal projection, and define
  $$b_N(\phi_N)= \Pi_N b(\phi_N, \phi_N), \quad G_N^{k,\alpha}(\phi_N)= \Pi_N(\sigma_{k,\alpha} \cdot\nabla \phi_N), \quad  \phi_N\in H_N. $$
Consider the finite dimensional SDE on $H_N$:
  \begin{equation}\label{approx-SDE}
  \d \phi_N(t) = \big[-b_N(\phi_N(t)) + \Delta\phi_N(t) \big]\,\d t + \frac{C_\nu}{\|\theta \|_{\ell^2}} \sum_{k,\alpha} \theta_k G_N^{k,\alpha}(\phi_N(t)) \circ \d W^{k,\alpha}_t, \quad \phi_N(0) = \Pi_N \Phi_0.
  \end{equation}
The a priori estimate in Lemma \ref{lem-apriori} tells us that, for any $N\geq 1$, $\P$-a.s. for all $t\leq T$,
  \begin{equation}\label{uniform-bounds}
  \|\phi_N(t)\|_{L^2} + \int_0^t \|\phi_N(s)\|_{H^1}^2\,\d s \leq \|\phi_N(0)\|_{L^2} + C_{\delta, R, T} \leq \|\Phi_0 \|_{L^2} + C_{\delta, R, T}.
  \end{equation}
Therefore, there exists a subsequence $\{\phi_{N_i} \}_{i\geq 1}$ converging weakly-$\ast$ in $L^\infty\big(\Omega, L^\infty(0,T; L^2)\big)$ and weakly in $L^2\big(\Omega, L^2(0,T; H^1)\big)$. To show the weak existence of solutions to \eqref{stoch-MHD-cut}, we shall use the classical compactness argument as in \cite{FlaGat95}; here, we follow more closely \cite[Section 3]{FlaLuo21}.

Let $Q_N$ be the law of the process $\phi_N(\cdot),\, N\geq 1$. By \eqref{uniform-bounds} and using the equations \eqref{approx-SDE}, it is not difficult to show (see e.g. \cite[Corollary 3.5]{FlaLuo21}) that there exists $C>0$ such that
  \begin{equation}\label{expectations.1}
  \sup_{N\geq 1} \bigg[\E \int_0^T \|\phi_N(t)\|_{H^1}^2 \,\d t + \E\int_0^T\int_0^T \frac{\|\phi_N(t) -\phi_N(s)\|_{H^{-6}}^2}{|t-s|^{1+2\gamma}} \,\d t\d s \bigg] \leq C,
  \end{equation}
where $\gamma\in (0,1/2)$ is fixed; moreover, for any $p$ big enough, there is $C_p>0$ such that
  \begin{equation}\label{expectations.2}
  \sup_{N\geq 1} \bigg[\E \int_0^T \|\phi_N(t)\|_{L^2}^p \,\d t + \E\int_0^T\int_0^T \frac{\|\phi_N(t) -\phi_N(s)\|_{H^{-6}}^4}{|t-s|^{7/3}} \,\d t\d s \bigg] \leq C_p.
  \end{equation}
Recall the compact embeddings (see \cite{Simon})
  \begin{equation}\label{Simon-embedding}
  \aligned
  L^2(0,T; V) \cap W^{\gamma,2}\big(0,T; H^{-6} \big) &\subset L^2(0,T; H),\\
  L^p(0,T; H) \cap W^{1/3,4}\big(0,T; H^{-6} \big) &\subset C\big([0,T], H^{-\delta} \big),
  \endaligned
  \end{equation}
where $W^{\alpha, p}\big(0,T; H^{-6} \big)$, for some $\alpha\in (0,1)$ and $ p>1$, is the time fractional Sobolev space.  The above uniform bounds imply that the family $\{Q_N\}_{N\geq 1}$ is tight in
  $$L^2(0,T; H) \quad\mbox{and} \quad C\big([0,T], H^{-\delta} \big). $$
Now, the Prohorov theorem (see \cite[p.59, Theorem 5.1]{Billingsley}) implies that there exists a subsequence $\{Q_{N_i} \}_{i\geq 1}$ which converges weakly to some probability measure $Q$ supported on $L^2(0,T; H) \cap C\big([0,T], H^{-\delta} \big)$. Moreover, by the Skorohod representation theorem (see \cite[p.70, Theorem 6.7]{Billingsley}), there exists a new probability space $\big(\tilde\Omega, \tilde{\mathcal F}, \tilde\P\big)$ and a sequence of stochastic processes $\big\{\tilde \phi_{N_i} \big\}_{i\geq 1}$ and $\tilde\phi$ defined on $\Omega$, such that
\begin{itemize}
\item[(i)] $\tilde \phi_{N_i}$ has the law $Q_{N_i}$ for all $i\geq 1$, and $\tilde\phi$ has the law $Q$;
\item[(ii)] $\tilde\P$-a.s., $\tilde \phi_{N_i}$ converges as $i\to \infty$ to $\tilde\phi$, in the topology of $L^2(0,T; H) \cap C\big([0,T], H^{-\delta} \big)$.
\end{itemize}
Since $\tilde \phi_{N_i}$ has the same law $Q_{N_i}$ as $\phi_{N_i}$ and the latter enjoys the pathwise estimate \eqref{uniform-bounds}, one can deduce that the limit $\tilde\phi$ satisfies also
  \begin{equation}\label{limit-bounds}
  \big\|\tilde\phi \big\|_{L^\infty(0,T; H)} \vee \big\|\tilde\phi \big\|_{L^2(0,T; V)} \leq C_{\|\Phi_0 \|_{H}, \delta, R, T}.
  \end{equation}
Having these preparations in mind and writing \eqref{approx-SDE} in the weak form, one can pass to the limit in the nonlinear terms and prove that $\tilde\phi$ solves \eqref{stoch-MHD-cut} in the weak sense. The details are omitted here; we only give a sketched proof of the fact that \eqref{stoch-MHD-cut} enjoys the pathwise uniqueness among those solutions satisfying the bounds \eqref{solution-property}.

\begin{proof}[Proof of pathwise uniqueness of \eqref{stoch-MHD-cut}]
First, we remark that, thanks to the bounds \eqref{solution-property}, it is enough to require that $v\in H^1(\T^3,\R^3)$ in the equations \eqref{defin.1} and \eqref{defin.2}; indeed, we have
  $$\int_0^t \< \xi_s,\Delta v + S_\theta(v) \> \,\d s= -\int_0^t \< \nabla\xi_s,\nabla v \> \,\d s - \frac{C_\nu^2}{\|\theta \|_{\ell^2}^2} \sum_{k,\alpha} \theta_k^2 \int_0^t \big\< \sigma_{k,\alpha}\cdot\nabla \xi_s, \Pi(\sigma_{-k, \alpha} \cdot\nabla v) \big\> \,\d s $$
and similarly for the corresponding term in \eqref{defin.2}. Here we write $\<\cdot,\cdot\>$ for the duality between $H^1(\T^3)$ and $H^{-1}(\T^3)$, which are used to denote $\R^d$-valued functions and distributions with $d=3$ or $6$.  Moreover, one can show that the assumptions of \cite[p.72, Theorem 2.13]{RozLot} are verified, thus we can apply the It\^o formula \cite[(2.5.3)]{RozLot}. In the following we omit the time variable to simplify notations.

Let $\Phi_i= (\xi_i, \eta_i)^\ast,\, i=1,2$ be two solutions to \eqref{stoch-MHD-cut} on the same filtered probability space $(\Omega, \mathcal F, (\mathcal F_t), \P)$, with the same initial condition $\Phi_0=(\xi_0, \eta_0)^\ast$ and the same family of Brownian motions $\{W^{k,\alpha} \}_{k\in \Z^3_0, \alpha=1,2}$, satisfying
  \begin{equation}\label{proof-uniq-bounds}
  \P \mbox{-a.s.}, \quad \|\Phi_i \|_{L^\infty(0,T; H)} \vee \|\Phi_i \|_{L^2(0,T; V)} \leq C_{\|\Phi_0 \|_{H}, \delta, R, T}, \quad i=1,2. \end{equation}
Then, it holds in the distribution sense that, for $i=1,2$,
  $$\d\Phi_i= -f_R(\Phi_i) b(\Phi_i, \Phi_i)\,\d t + \big[\Delta\Phi_i + S_\theta(\Phi_i) \big]\,\d t + \frac{C_\nu}{\|\theta \|_{\ell^2}} \sum_{k, \alpha} \theta_k \Pi(\sigma_{k,\alpha} \cdot\nabla \Phi_i)\, \d W^{k,\alpha}. $$
Let $\Phi = (\xi, \eta)^\ast := \Phi_1 -\Phi_2 = (\xi_1 -\xi_2, \eta_1 -\eta_2)^\ast$; then,
  $$\aligned
  \d \Phi=&\, - \big[f_R(\Phi_1) b(\Phi_1, \Phi_1) -f_R(\Phi_2) b(\Phi_2, \Phi_2)\big]\,\d t + \big[\Delta\Phi + S_\theta(\Phi) \big]\,\d t \\
  &\, + \frac{C_\nu}{\|\theta \|_{\ell^2}} \sum_{k, \alpha} \theta_k \Pi(\sigma_{k,\alpha} \cdot\nabla \Phi)\, \d W^{k,\alpha}.
  \endaligned $$
By the It\^o formula (see \cite[(2.5.3)]{RozLot}),
  $$\aligned
  \d\|\Phi \|_{L^2}^2 =&\, -2\big\<\Phi, f_R(\Phi_1) b(\Phi_1, \Phi_1) -f_R(\Phi_2) b(\Phi_2, \Phi_2) \big>\,\d t + 2\<\Phi, \Delta\Phi + S_\theta(\Phi) \> \,\d t \\
  &\, + \frac{2 C_\nu}{\|\theta \|_{\ell^2}} \sum_{k, \alpha} \theta_k \<\Phi, \Pi(\sigma_{k,\alpha} \cdot\nabla \Phi)\>\, \d W^{k,\alpha} + \frac{2C_\nu^2}{\|\theta \|_{\ell^2}^2} \sum_{k, \alpha} \theta_k^2 \|\Pi(\sigma_{k,\alpha} \cdot\nabla \Phi) \|_{L^2}^2 \,\d t .
  \endaligned $$
The definition of $S_\theta(\Phi)$ leads to
  $$\<\Phi, \Delta\Phi + S_\theta(\Phi) \> = - \|\nabla \Phi\|_{L^2}^2 - \frac{C_\nu^2}{\|\theta \|_{\ell^2}^2} \sum_{k, \alpha} \theta_k^2 \|\Pi(\sigma_{k,\alpha} \cdot\nabla \Phi) \|_{L^2}^2. $$
Moreover, since $\sigma_{k,\alpha}$ and the components $\xi,\eta$ of $\Phi$ are all divergence free, we have
  $$\<\Phi, \Pi(\sigma_{k, \alpha} \cdot\nabla \Phi) \> = \<\Phi, \sigma_{k, \alpha} \cdot\nabla \Phi \> =0 . $$
Consequently,
  \begin{equation}\label{proof-unique.1}
  \d\|\Phi \|_{L^2}^2 = -2\big\<\Phi, f_R(\Phi_1) b(\Phi_1, \Phi_1) -f_R(\Phi_2) b(\Phi_2, \Phi_2) \big\>\,\d t - 2\|\nabla \Phi\|_{L^2}^2 \,\d t.
  \end{equation}

It remains to estimate the first term on the right hand side of \eqref{proof-unique.1}. We have
  \begin{equation}\label{proof-unique.1.5}
  \aligned
  &\quad\ \big| \big\<\Phi, f_R(\Phi_1) b(\Phi_1, \Phi_1) -f_R(\Phi_2) b(\Phi_2, \Phi_2) \big\> \big| \\
  &\leq \big| (f_R(\Phi_1) -f_R(\Phi_2)) \<\Phi, b(\Phi_1, \Phi_1) \> \big| + f_R(\Phi_2) \big| \big\<\Phi, b(\Phi_1, \Phi_1) -b(\Phi_2, \Phi_2) \big\> \big| \\
  &=:  J_1 + J_2.
  \endaligned
  \end{equation}
First, it is clear that
  \begin{equation}\label{proof-unique.2}
  \aligned
  |f_R(\Phi_1) -f_R(\Phi_2)| &\leq \|f'_R \|_\infty \big| \|\Phi_1 \|_{H^{-\delta}} - \|\Phi_2 \|_{H^{-\delta}} \big|\leq C \|\Phi_1 -\Phi_2 \|_{H^{-\delta}} \leq C \|\Phi \|_{L^2},
  \endaligned
  \end{equation}
and, by definition,
  $$\<\Phi, b(\Phi_1, \Phi_1) \>= \<\xi, \L_{u_1} \xi_1 - \L_{B_1}\eta_1\> + \<\eta, \L_{u_1} \eta_1 - \L_{B_1}\xi_1 -2T(B_1, u_1)\> ,$$
thus
  \begin{equation}\label{proof-unique.3}
  J_1\leq C \|\Phi \|_{L^2} \Big( |\<\xi, \L_{u_1} \xi_1 \>| + |\<\xi, \L_{B_1}\eta_1\>| +|\<\eta, \L_{u_1} \eta_1 \>| + |\<\eta, \L_{B_1}\xi_1\>|  + |\<\eta, T(B_1, u_1) \>| \Big).
  \end{equation}

We demonstrate how to estimate the first term on the right hand side and the others can be treated in a similar way. We have
  $$J_{1,1}:= C \|\Phi \|_{L^2} |\<\xi, \L_{u_1} \xi_1 \>| \leq C \|\Phi \|_{L^2} |\<\xi, u_1\cdot\nabla \xi_1 \>| + C \|\Phi \|_{L^2} |\<\xi, \xi_1 \cdot\nabla u_1\>| =: J_{1,1,1} + J_{1,1,2}. $$
By H\"older's inequality with exponent $\frac13 +\frac16 + \frac12 =1$, we have
  \begin{equation*}
  J_{1,1,1} \leq C \|\Phi \|_{L^2} \|\xi \|_{L^3} \|u_1\|_{L^6} \|\nabla\xi_1 \|_{L^2} \leq C \|\Phi \|_{L^2} \|\xi \|_{H^{1/2}} \|u_1\|_{H^1} \|\nabla\xi_1 \|_{L^2} ,
  \end{equation*}
where we have used the Sobolev embedding inequalities \eqref{Sobolev-embedding}. Moreover, applying the interpolation inequality and the Poincar\'e inequality,
  \begin{equation}\label{pathwise-uniq-3}
  J_{1,1,1} \leq C\|\Phi \|_{L^2} \|\xi \|_{L^2}^{1/2} \|\xi \|_{H^1}^{1/2} \|\xi_1 \|_{L^2} \|\nabla\xi_1 \|_{L^2} \leq C \|\Phi \|_{L^2}^{3/2} \|\nabla \xi \|_{L^2}^{1/2} \|\xi_1 \|_{L^2} \|\nabla\xi_1 \|_{L^2}.
  \end{equation}
By Young's inequality with exponent $\frac14 + \frac34 =1$, for $\eps>0$ small enough,
  $$J_{1,1,1} \leq \eps \|\nabla \xi \|_{L^2}^2 + C_\eps \|\Phi \|_{L^2}^{2} \|\xi_1 \|_{L^2}^{4/3} \|\nabla\xi_1 \|_{L^2}^{4/3} \leq \eps \|\nabla \xi \|_{L^2}^2 + C_\eps \|\Phi \|_{L^2}^{2} \|\nabla\xi_1 \|_{L^2}^{4/3}, $$
since, by \eqref{proof-uniq-bounds}, $\xi_1$ is a.s. bounded in $ L^\infty(0,T; L^2)$. Next we turn to estimate $J_{1,1,2}$. By H\"older's inequality with exponent $\frac13 + \frac13 +\frac13 =1$,
  $$\aligned
  J_{1,1,2} &\leq C \|\Phi \|_{L^2} \|\xi \|_{L^3} \|\xi_1 \|_{L^3} \|\nabla u_1\|_{L^3} \leq C \|\Phi \|_{L^2} \|\xi \|_{H^{1/2}} \|\xi_1 \|_{H^{1/2}}^2 \\
  &\leq C \|\Phi \|_{L^2} \|\xi \|_{L^2}^{1/2} \|\xi \|_{H^1}^{1/2} \|\xi_1 \|_{L^2} \|\xi_1 \|_{H^1} \leq C \|\Phi \|_{L^2}^{3/2} \|\nabla \xi \|_{L^2}^{1/2} \|\xi_1 \|_{L^2} \|\nabla\xi_1 \|_{L^2} ,
  \endaligned $$
which is the same as the right hand side of \eqref{pathwise-uniq-3}. Thus, similarly as above, we have
  $$J_{1,1,2} \leq \eps \|\nabla \xi \|_{L^2}^2 + C_\eps \|\Phi \|_{L^2}^{2} \|\nabla\xi_1 \|_{L^2}^{4/3}. $$
Summarizing the above arguments, we obtain
  \begin{equation}\label{J-1}
  J_{1,1} \leq 2 \eps \|\nabla \xi \|_{L^2}^2 + C_\eps \|\Phi \|_{L^2}^{2} \|\nabla\xi_1 \|_{L^2}^{4/3}.
  \end{equation}
Proceeding as above for other terms in \eqref{proof-unique.3}, we finally get
  $$J_1\leq n_1\eps \|\nabla \Phi \|_{L^2}^2 + C_\eps \|\Phi \|_{L^2}^2 \big(\|\nabla\Phi_1 \|_{L^2}^{4/3} + \|\nabla\Phi_1 \|_{L^2}^2 +  \|\nabla\Phi_2 \|_{L^2}^{4/3} + \|\nabla\Phi_2 \|_{L^2}^2 \big),$$
where $n_1$ is some integer and $\|\nabla \Phi \|_{L^2}^2= \|\nabla \xi \|_{L^2}^2 + \|\nabla \eta \|_{L^2}^2$. The estimate of $J_2$ in \eqref{proof-unique.1.5} is similar, and thus we finally get, for some $n\in \N$,
  $$\aligned
  &\quad\ \big| \big\<\Phi, f_R(\Phi_1) b(\Phi_1, \Phi_1) -f_R(\Phi_2) b(\Phi_2, \Phi_2) \big\> \big| \\
  &\leq n \eps \|\nabla \Phi \|_{L^2}^2 + C_\eps \|\Phi \|_{L^2}^2 \big(\|\nabla\Phi_1 \|_{L^2}^{4/3} + \|\nabla\Phi_1 \|_{L^2}^2 +  \|\nabla\Phi_2 \|_{L^2}^{4/3} + \|\nabla\Phi_2 \|_{L^2}^2 \big).
  \endaligned $$
Taking $\eps = 1/n$ and substituting this estimate into \eqref{proof-unique.1}, we arrive at
  $$\d\|\Phi \|_{L^2}^2 \leq C_n \|\Phi \|_{L^2}^2 \big(\|\nabla\Phi_1 \|_{L^2}^{4/3} + \|\nabla\Phi_1 \|_{L^2}^2 +  \|\nabla\Phi_2 \|_{L^2}^{4/3} + \|\nabla\Phi_2 \|_{L^2}^2 \big)\,\d t.  $$
The bounds in \eqref{proof-uniq-bounds} implies that the quantity in the bracket on the right hand side is integrable with respect to $t\in [0,T]$. Since $\Phi_0=0$, Gronwall's inequality yields, $\P$-a.s., $\Phi_t=0$ for all $t\in [0,T]$. Thus we have proved the pathwise uniqueness property of \eqref{stoch-MHD-cut}.
\end{proof}

\section{Scaling limit and proofs of the main result} \label{sec-proof}

Recall the sequences $\{\theta^N \}_{N\geq 1} \subset \ell^2$ defined in \eqref{theta-N}; we consider the following sequence of stochastic 3D MHD equations with cut-off:
  \begin{equation}\label{stoch-MHDEs}
  \d \Phi^N + f_R(\Phi^N) b(\Phi^N, \Phi^N)\,\d t = \big[\Delta \Phi^N + S_{\theta^N}(\Phi^N) \big]\,\d t + \frac{C_\nu}{\|\theta^N \|_{\ell^2}} \sum_{k,\alpha} \theta^N_k \Pi(\sigma_{k,\alpha} \cdot\nabla \Phi^N)\, \d W^{k,\alpha}
  \end{equation}
subject to the initial data $\Phi^N_0= \big(\xi^N_0, \eta^N_0 \big)^\ast \in B_H(K)$. For any $T>0$, Theorem \ref{thm-existence} implies that there exists a pathwise unique strong solution $\Phi^N_t= \big(\xi^N_t, \eta^N_t \big)^\ast $ to \eqref{stoch-MHDEs} satisfying
  \begin{equation}\label{bounds}
  \|\Phi^N \|_{L^\infty(0,T; H)} \vee \|\Phi^N \|_{L^2(0,T; V)} \leq C_{K,\delta, R, T},
  \end{equation}
where $ C_{K,\delta, R, T}$ is some constant independent of $\nu$ and $N$; furthermore, for any divergence free test vector field $v$, the following equalities hold $\P$-a.s. on $[0,T]$:
  \begin{equation}\label{solu.1}
  \aligned
  \< \xi^N_t,v\> =&\, \<\xi^N_0, v\>+ \int_0^t f_R(\Phi^N_s) \Big[\big\< \xi^N_s,\L_{u^N_s}^{\ast}v \big\> - \big\< \eta^N_s,\L_{B^N_s}^{\ast}v \big\> \Big] \,\d s\\
  &\, + \int_0^t \< \xi^N_s,\Delta v + S_{\theta^N}(v) \>\,\d s -\frac{C_{\nu}}{\| \theta^N\|_{\ell^2}} \sum_{k,\alpha}\theta^N_k \int_0^t \< \xi^N_s, \sigma_{k,\alpha}\cdot \nabla v \> \,\d W_{s}^{k,\alpha},
  \endaligned
  \end{equation}
  \begin{equation}\label{solu.2}
  \aligned
  \< \eta^N_t,v\> =&\, \<\eta^N_0, v\>+ \int_0^t f_R(\Phi^N_s) \Big[ \big\< \eta^N_s,\L_{u^N_s}^{\ast}v \big\> - \big\< \xi^N_s,\L_{B^N_s}^\ast v \big\> -2 \big<T(B^N_s,u^N_s), v\big> \Big] \,\d s\\
  &\, + \int_0^t \< \eta^N_s,\Delta v + S_{\theta^N}(v) \>\,\d s -\frac{C_{\nu}}{\| \theta^N\|_{\ell^2}} \sum_{k,\alpha}\theta_{k} \int_0^t \< \eta^N_s, \sigma_{k,\alpha}\cdot \nabla v \> \,\d W_{s}^{k,\alpha} .
  \endaligned
  \end{equation}
The first main result of this section is the next scaling limit theorem.

\begin{theorem}[Scaling limit] \label{thm-limit}
Fix any $K>0$ and $T>0$, and assume that the initial data $\Phi^N_0$ converge weakly in $H$ to some $\Phi_0$. Then we can find big $\nu$ and $R$ such that, as $N\to \infty$, the solutions $\Phi^N$ converge in probability to the unique global solution $\Phi$ of the deterministic 3D MHD equations
  \begin{equation}\label{thm-limit.1}
  \partial_t \Phi + b(\Phi, \Phi) = \Big(1+ \frac35 \nu\Big) \Delta \Phi.
  \end{equation}
Moreover, denoting by $\|\cdot \|_{\mathcal X}= \|\cdot \|_{C([0,T], H^{-\delta})} \vee \|\cdot \|_{L^2(0,T;H)}$, then for any $\eps>0$,
  \begin{equation}\label{thm-limit.2}
  \lim_{N\to \infty} \sup_{\Phi_0\in B_H(K)} \P \big(\|\Phi^N(\cdot, \Phi_0) - \Phi(\cdot, \Phi_0) \|_{\mathcal X} > \eps \big) =0,
  \end{equation}
where we write $\Phi^N(\cdot, \Phi_0)$ (resp. $\Phi(\cdot, \Phi_0)$) for the unique solution to \eqref{stoch-MHDEs} (resp. \eqref{thm-limit.1}) with the initial condition $\Phi_0$.
\end{theorem}

\iffalse
We want to identify the limit of the sequence of solutions $\Phi_N$. Similarly to the discussions above \eqref{thm-limit.1}, the key limit \eqref{theta-N.2} implies that the viscosity coefficient in the limit equation will be $1+\frac35 \nu$; moreover, by \eqref{theta-N.1} and \eqref{solu-bounds}, one can show that the martingale part in \eqref{stoch-MHD-cut-N} vanishes in the limit procedure. Let $\Phi$ be any limit process of $\{\Phi_N \}_{N\geq 1}$; we deduce that it solves the deterministic 3D MHD equations with cut-off and enhanced dissipation:
  $$\partial_t \Phi + f_R(\Phi) b(\Phi, \Phi) = \Big(1+ \frac35 \nu\Big) \Delta \Phi. $$
For fixed $K>0$, if we take $\nu$ and $R$ big enough, then it is well known that, for any initial condition $\Phi_0\in B_H(K)$, the above equations admit a unique global solution $\Phi(t,\Phi_0)$; moreover, $\|\Phi(t,\Phi_0) \|_{H^{-\delta}} <R$ for all $t>0$, which means the cut-off $f_R(\Phi)$ is identically 1 and thus can be omitted. To sum up, we obtain the main result of this paper.
\fi

We follow some of the arguments below the proof of Lemma \ref{lem-apriori}.  Let $Q^N$ be the law of $(\Phi^N_t)_{t\in [0,T]}, \, N \geq 1$. Using the uniform bounds \eqref{bounds} and the equations \eqref{solu.1} and \eqref{solu.2}, we can prove similar estimates as \eqref{expectations.1} and \eqref{expectations.2}. Therefore, the family $\{Q^N\}_{N\geq 1}$ of laws is tight on $L^2(0,T; H)$ and on $C\big([0,T], H^{-\delta} \big)$. Thus, by the Prohorov theorem we can find a subsequence $\{Q^{N_i} \}_{i\geq 1}$ converging weakly to some probability measure $Q$, supported on $L^2(0,T; H)$ and on $C\big([0,T], H^{-\delta} \big)$. Applying the Skorohod theorem yields a probability space $\big(\hat\Omega, \hat{\mathcal F}, \hat\P\big)$ and a sequence of processes $\{\hat\Phi^{N_i} \}_{i\geq 1}$ and $\hat\Phi$ defined on $\hat\Omega$, such that
\begin{itemize}
\item[(a)] $\hat\Phi$ is distributed as $Q$ and $\hat\Phi^{N_i}$ is distributed as $Q^{N_i}$ for all $i\geq 1$;
\item[(b)] $\hat\P$-a.s., $\hat\Phi^{N_i}$ converges as $i\to \infty$ to $\hat\Phi$ in the topology of $L^2(0,T; H)$ and on $C\big([0,T], H^{-\delta} \big)$.
\end{itemize}
We remark that, for every $N\geq 1$, if we consider $Q^N$ together with the laws of the Brownian motions $\{W^{k,\alpha}: k\in \Z^3_0, \alpha=1,2 \}$, then we can find on the new probability space $\hat\Omega$ a sequence of Brownian motions $\{\hat W^{N_i, k,\alpha}: k\in \Z^3_0, \alpha=1,2 \}_{i\geq 1} $, such that, for any $i\geq 1$, $\hat\Phi^{N_i}$ and $\hat W^{N_i, k,\alpha},\, k\in \Z^3_0,\, \alpha=1,2$ satisfy the equations \eqref{solu.1} and \eqref{solu.2}. Let $\hat u^{N_i}$ and $\hat B^{N_i}$ be the velocity field and magnetic field on the new probability space $\hat\Omega$, corresponding to the solution $\hat\Phi^{N_i}= (\hat\xi^{N_i}, \hat\eta^{N_i})$; then,
  \begin{equation}\label{solu.3}
  \aligned
  \< \hat\xi^{N_i}_t,v\> =&\, \<\xi^{N_i}_0, v\>+ \int_0^t f_R(\hat\Phi^{N_i}_s) \Big[ \big\< \hat \xi^{N_i}_s, \L_{\hat u^{N_i}_s}^{\ast}v \big\> - \big\< \hat\eta^{N_i}_s, \L_{\hat B^{N_i}_s}^{\ast}v \big\> \Big] \,\d s\\
  &\, + \int_0^t \< \hat\xi^{N_i}_s, \Delta v + S_{\theta^{N_i}}(v) \> \,\d s -\frac{C_{\nu}}{\| \theta^{N_i}\|_{\ell^2}} \sum_{k,\alpha}\theta^{N_i}_k \int_0^t \< \hat\xi^{N_i}_s, \sigma_{k,\alpha}\cdot \nabla v \> \,\d \hat W_{s}^{N_i, k,\alpha},
  \endaligned
  \end{equation}
  \begin{equation}\label{solu.4}
  \aligned
  \< \hat\eta^{N_i}_t,v\> =&\, \<\eta^{N_i}_0, v\>+ \int_0^t f_R(\hat\Phi^{N_i}_s) \Big[ \big\< \hat\eta^{N_i}_s, \L_{\hat u^{N_i}_s}^{\ast}v \big\> - \big\< \hat\xi^{N_i}_s, \L_{\hat B^{N_i}_s}^\ast v \big\> -2 \big<T(\hat B^{N_i}_s, \hat u^{N_i}_s), v\big> \Big] \,\d s\\
  &\, + \int_0^t \< \hat\eta^{N_i}_s,\Delta v + S_{\theta^{N_i}}(v) \>\,\d s -\frac{C_{\nu}}{\| \theta^{N_i} \|_{\ell^2}} \sum_{k,\alpha}\theta_{k} \int_0^t \< \hat\eta^{N_i}_s, \sigma_{k,\alpha}\cdot \nabla v \> \,\d \hat W_{s}^{N_i,k,\alpha} .
  \endaligned
  \end{equation}
With these results in mind, we can prove

\begin{lemma}\label{lem-limit}
Let $\Phi_0$ be the weak limit of the initial data $\{\Phi^N_0\}_{N\geq 1}$; then the limit process $\hat\Phi$ solves in the weak sense the deterministic 3D MHD equations with cut-off:
  \begin{equation}\label{lem-limit.1}
  \partial_t \hat\Phi + f_R(\hat\Phi) b(\hat\Phi, \hat\Phi) = \Big(1+ \frac35 \nu\Big) \Delta \hat\Phi, \quad \hat\Phi|_{t=0} = \Phi_0.
  \end{equation}
\end{lemma}

\begin{proof}
Let $\hat u$ and $\hat B$ be the velocity field and magnetic field associated to the limit process $\hat\Phi= (\hat\xi, \hat\eta)$. Thanks to item (b) above, it is standard to prove the convergence of the terms in the first line of \eqref{solu.3}. Moreover, by the key limit \eqref{theta-N.2}, it is clear that, $\hat\P$-a.s., as $i\to \infty$,
  $$\int_0^\cdot \< \hat\xi^{N_i}_s, \Delta v + S_{\theta^{N_i}}(v) \> \,\d s \to \Big(1+ \frac35 \nu\Big) \int_0^\cdot \< \hat\xi_s, \Delta v \> \,\d s$$
in the topology of $C([0,T],\R)$.

Next, we deal with the martingale part: by the It\^o isometry,
  $$\aligned
  &\ \hat\E \bigg[\Big|\frac{C_{\nu}}{\| \theta^{N_i}\|_{\ell^2}} \sum_{k,\alpha}\theta^{N_i}_k \int_0^t \big\< \hat\xi^{N_i}_s, \sigma_{k,\alpha}\cdot \nabla v \big\> \,\d \hat W_{s}^{N_i, k,\alpha} \Big|^2 \bigg] \\
  = &\ \hat\E \bigg[ \frac{C_{\nu}^2}{\| \theta^{N_i}\|_{\ell^2}^2} \sum_{k,\alpha} (\theta^{N_i}_k)^2 \int_0^t \big\< \hat\xi^{N_i}_s, \sigma_{k,\alpha}\cdot \nabla v \big\>^2 \,\d s \bigg] \\
  \leq &\ C_{\nu}^2\frac{\| \theta^{N_i}\|_{\ell^\infty}^2}{\| \theta^{N_i}\|_{\ell^2}^2}\, \hat\E\int_0^t \sum_{k,\alpha} \big\< \hat\xi^{N_i}_s, \sigma_{k,\alpha}\cdot \nabla v \big\>^2 \,\d s.
  \endaligned $$
Note that $\hat\Phi^{N_i}= (\hat\xi^{N_i}, \hat\eta^{N_i})$ has the same law as the solution $\Phi^{N_i}$ to \eqref{stoch-MHDEs}, thus it fulfils the estimates \eqref{bounds}. We have
  $$\sum_{k,\alpha} \big\< \hat\xi^{N_i}_s, \sigma_{k,\alpha}\cdot \nabla v \big\>^2= \sum_{k,\alpha} \big\< (\nabla v)^\ast \hat\xi^{N_i}_s, \sigma_{k,\alpha} \big\>^2 \leq \big\|(\nabla v)^\ast \hat\xi^{N_i}_s \big\|_{L^2}^2 \leq C_{K,\delta, R, T} \|\nabla v\|_\infty^2, $$
where in the second step we have used the fact that $\{\sigma_{k,\alpha} \}_{k\in \Z^3_0, \alpha=1,2}$ is an orthonormal family in $L^2(\T^3, \R^3)$. Consequently,
  $$\hat\E \bigg[\Big|\frac{C_{\nu}}{\| \theta^{N_i}\|_{\ell^2}} \sum_{k,\alpha}\theta^{N_i}_k \int_0^t \big\< \hat\xi^{N_i}_s, \sigma_{k,\alpha}\cdot \nabla v \big\> \,\d \hat W_{s}^{N_i, k,\alpha} \Big|^2 \bigg] \leq \ C_{\nu}^2 C_{K,\delta, R, T} \|\nabla v\|_\infty^2 \frac{\| \theta^{N_i}\|_{\ell^\infty}^2}{\| \theta^{N_i}\|_{\ell^2}^2}. $$
The right hand side vanishes as $i\to \infty$ due to the limit \eqref{theta-N.1}, thus the martingale part tends to 0 in the mean square sense. Summarizing the above arguments, we have proved
  $$\< \hat\xi_t,v\> = \<\xi_0, v\>+ \int_0^t f_R(\hat\Phi_s) \Big[ \big\< \hat \xi_s, \L_{\hat u_s}^{\ast}v \big\> - \big\< \hat\eta_s, \L_{\hat B_s}^{\ast}v \big\> \Big] \,\d s + \Big(1+ \frac35 \nu\Big) \int_0^t \< \hat\xi_s, \Delta v\> \,\d s . $$
In the same way, we can prove the second equation:
  $$\aligned
  \< \hat\eta_t,v\> =&\, \<\eta_0, v\>+ \Big(1+ \frac35 \nu\Big) \int_0^t \< \hat\eta_s,\Delta v \> \,\d s\\
  &\, + \int_0^t f_R(\hat\Phi_s) \Big[ \big\< \hat\eta_s, \L_{\hat u_s}^{\ast}v \big\> - \big\< \hat\xi_s, \L_{\hat B_s}^\ast v \big\> -2 \big<T(\hat B_s, \hat u_s), v\big> \Big] \,\d s.
  \endaligned $$
This finishes the proof.
\end{proof}

Next we give a short proof of the following well known result.

\begin{lemma}\label{lem-3D-MHD}
Given $K>0$, there exists a big $\nu=\nu(K)>0$ such that for all $\Phi_0\in B_H(K)$, the deterministic 3D MHD equations \eqref{thm-limit.1} admit a unique global solution $\Phi(\cdot,\Phi_0)$; moreover, there exist some $C_K>0$ and $\lambda=\lambda(K)>0$ such that for any $\Phi_0 \in B_H(K)$, it holds
  \begin{equation}\label{exponential-decay}
  \|\Phi(t,\Phi_0) \|_{L^2} \leq C_K\, e^{-\lambda t}, \quad t\geq 0.
  \end{equation}
\end{lemma}

\begin{proof}
The proof is simpler than that of Lemma \ref{lem-apriori}. We denote by $\nu_1= 1+ \frac35 \nu$; then,
  $$\frac{\d}{\d t} \|\Phi\|_{L^2}^2 = -2 \<\Phi, b(\Phi, \Phi)\> - 2\nu_1 \|\nabla \Phi\|_{L^2}^2. $$
More precisely,
  \begin{equation}\label{lem-3D-MHD.1}
  \frac{\d}{\d t} \big(\|\xi \|_{L^2}^2 + \|\eta \|_{L^2}^2 \big) = -2\big(\<\xi, b_1(\Phi, \Phi)\> + \<\eta, b_2(\Phi, \Phi)\>\big) -2\nu_1 \big(\|\nabla \xi \|_{L^2}^2 + \|\nabla \eta \|_{L^2}^2\big) .
  \end{equation}
We will estimate separately the first two quantities on the right-hand side.

\emph{Step 1.} By the definition of $b_1(\Phi, \Phi)$, we have
  $$\<\xi, b_1(\Phi, \Phi)\> = \<\xi, \L_u \xi \> - \<\xi, \L_B \eta \> =: I_1 + I_2. $$
First, since $u$ is divergence free, we have $I_1= -\<\xi, \xi\cdot\nabla u \>$ and by H\"older's inequality, Sobolev embedding and interpolation inequality,
  $$|I_1| \leq \|\xi \|_{L^3}^2 \|\nabla u \|_{L^3} \leq C\|\xi \|_{L^3}^3 \leq C\|\xi \|_{H^{1/2}}^3 \leq C\|\xi \|_{L^2}^{3/2} \|\xi \|_{H^1}^{3/2} \leq \eps \|\xi \|_{H^1}^2 + C^4 \eps^{-3} \|\xi \|_{L^2}^6, $$
where  $\eps>0$ is a small constant. Next,
  $$I_2= - \<\xi, B\cdot\nabla \eta \> + \<\xi, \eta\cdot\nabla B \>= :I_{2,1} + I_{2,2}. $$
For $I_{2,1}$, again by H\"older's inequality,
  $$|I_{2,1}| \leq \|\xi\|_{L^3} \|B\|_{L^6} \|\nabla \eta\|_{L^2}\leq C \|\xi \|_{H^{1/2}} \|B \|_{H^1} \|\eta\|_{H^1} \leq C\|\xi \|_{L^2}^{1/2} \|\xi \|_{H^1}^{1/2} \|\eta \|_{L^2} \|\eta\|_{H^1}. $$
By Young's inequality with exponents $\frac1{12} + \frac14 + \frac16 + \frac12 =1$, we obtain
  $$|I_{2,1}| \leq \eps\big(\|\xi \|_{H^1}^2 + \|\eta\|_{H^1}^2\big)+ C^4 \eps^{-3} \big(\|\xi \|_{L^2}^6 + \|\eta\|_{L^2}^6\big).$$
Similarly, for $I_{2,2}$, we have
  $$|I_{2,2}| \leq \|\xi\|_{L^3} \|\eta \|_{L^3} \|\nabla B\|_{L^3} \leq C \|\xi\|_{L^3} \|\eta \|_{L^3}^2 \leq C\|\xi \|_{L^2}^{1/2} \|\xi \|_{H^1}^{1/2} \|\eta \|_{L^2} \|\eta\|_{H^1} $$
and thus it admits the same estimate as $|I_{2,1}|$. In summary, we obtain
  $$|\<\xi, b_1(\Phi, \Phi)\>| \leq 3\eps\big(\|\xi \|_{H^1}^2 + \|\eta\|_{H^1}^2\big)+ C^4 \eps^{-3} \big(\|\xi \|_{L^2}^6 + \|\eta\|_{L^2}^6 \big).$$

\emph{Step 2.} Now we estimate
  $$\<\eta, b_2(\Phi, \Phi)\>= \<\eta, \L_u\eta \> - \<\eta, \L_B \xi\> -2 \<\eta, T(B,u)\>=: J_1+ J_2+ J_3. $$
The first term $J_1= -\<\eta, \eta\cdot\nabla u \>$ can be treated as follows:
  $$|J_1| \leq \|\eta \|_{L^3}^2 \|\nabla u \|_{L^3} \leq C \|\xi\|_{L^3} \|\eta \|_{L^3}^2 \leq \eps\big(\|\xi \|_{H^1}^2 + \|\eta\|_{H^1}^2\big)+ C^4 \eps^{-3} \big(\|\xi \|_{L^2}^6 + \|\eta\|_{L^2}^6 \big). $$
The second one $J_2$ is similar to $I_2$, and we have
  $$|J_2| \leq 2\eps\big(\|\xi \|_{H^1}^2 + \|\eta\|_{H^1}^2\big)+ C^4 \eps^{-3} \big(\|\xi \|_{L^2}^6 + \|\eta\|_{L^2}^6 \big).$$
Finally, by the definition of $T(B,u)$, we have
  $$\aligned
  |J_3| &= 2|\<\eta, T(B,u)\>| \leq C\|\eta \|_{L^3} \|\nabla B \|_{L^3}\|\nabla u \|_{L^3} \leq C\|\xi \|_{L^3} \|\eta \|_{L^3}^2 \\
  &\leq \eps\big(\|\xi \|_{H^1}^2 + \|\eta\|_{H^1}^2\big)+ C^4 \eps^{-3} \big(\|\xi \|_{L^2}^6 + \|\eta\|_{L^2}^6 \big).
  \endaligned $$
Summarizing these estimates leads to
  $$|\<\eta, b_2(\Phi, \Phi)\>| \leq 4\eps\big(\|\xi \|_{H^1}^2 + \|\eta\|_{H^1}^2\big)+ C^4 \eps^{-3} \big(\|\xi \|_{L^2}^6 + \|\eta\|_{L^2}^6 \big) $$
for another constant $C>0$ independent of $\eps$.

\emph{Step 3.} Substituting the estimates in \emph{Steps 1 and 2} into \eqref{lem-3D-MHD.1}, we arrive at
  $$\frac{\d}{\d t} \big(\|\xi \|_{L^2}^2 + \|\eta \|_{L^2}^2 \big) \leq - 2\nu_1 \big(\|\nabla \xi \|_{L^2}^2 + \|\nabla \eta \|_{L^2}^2\big) + 14 \eps\big(\|\xi \|_{H^1}^2 + \|\eta\|_{H^1}^2\big)+ C^4 \eps^{-3} \big(\|\xi \|_{L^2}^6 + \|\eta\|_{L^2}^6 \big); $$
written in more compact form, one has
  $$\frac{\d}{\d t} \|\Phi \|_{L^2}^2 \leq -2\nu_1 \|\nabla \Phi\|_{L^2}^2 + 14 \eps \|\Phi \|_{H^1}^2 + C^4 \eps^{-3} \|\Phi \|_{L^2}^6. $$
Choosing $\eps =\nu_1/14$ gives us
  \begin{equation}\label{lem-3D-MHD.2}
  \aligned
  \frac{\d}{\d t} \|\Phi \|_{L^2}^2 &\leq -\nu_1 \|\nabla \Phi\|_{L^2}^2 + \nu_1 \|\Phi \|_{L^2}^2 + C_1^4 \nu_1^{-3} \|\Phi \|_{L^2}^6 \\
  &\leq -(4\pi^2 -1) \nu_1 \|\Phi \|_{L^2}^2 +C_1^4 \nu_1^{-3} \|\Phi \|_{L^2}^6,
  \endaligned
  \end{equation}
where $C_1>0$ is another constant independent of $\nu_1 = 1+\frac35 \nu$ and in the second step we have used Poincar\'e's inequality on $\T^3$. This differential inequality can be solved explicitly and for $\|\Phi_0 \|_{L^2}\leq K$, if $\nu_1 = 1+\frac35 \nu$ is big enough such that
  $$\nu_1 \geq \frac{C_1 K}{\sqrt{\pi}}, $$
then one has
  $$\|\Phi(t,\Phi_0) \|_{L^2} \leq 2^{1/4} K e^{-(4\pi^2-1)\nu_1 t/2}, \quad \mbox{for all } t>0.$$
Thus we obtain \eqref{exponential-decay} with $C_K= 2^{1/4} K$ and $\lambda= (4\pi^2-1)\nu_1/2 >0$.
\end{proof}

Now we are ready to prove Theorem \ref{thm-limit}.

\begin{proof}[Proof of Theorem \ref{thm-limit}]
For fixed $K>0$, we choose $\nu= \nu(K)>0$ as in Lemma \ref{lem-3D-MHD}; if we take $R= C_K+1$, then the solution $\hat\Phi$ obtained in Lemma \ref{lem-limit} will not attain the cut-off threshold, and thus it coincides with the unique global solution $\Phi(\cdot,\Phi_0)$ of \eqref{thm-limit.1}. We conclude that the limit law $Q= \delta_{\Phi(\cdot,\Phi_0)}$, which is uniquely determined. The tightness of the family $\{Q^N\}_{N\geq 1}$ yields that the whole sequence converges weakly to $\delta_{\Phi(\cdot,\Phi_0)}$ as $N\to \infty$. As the limit is deterministic, we see that the processes $\{\Phi^N \}_{N\geq 1}$ converges in probability to $\Phi(\cdot,\Phi_0)$. This proves the first assertion of Theorem \ref{thm-limit}.

Next we turn to proving the second assertion of Theorem \ref{thm-limit}. We follow the idea of proof of \cite[Theorem 1.4]{FlaLuo21} and argue by contradiction. Suppose there exists $\eps_0>0$ small enough such that
  $$ \limsup_{N\to \infty} \sup_{\Phi_0\in B_H(K)} \P \big(\|\Phi^N(\cdot, \Phi_0) -\Phi (\cdot, \Phi_0) \|_{\mathcal X} >\eps_0 \big)>0, $$
where we have denoted by $\|\cdot \|_{\mathcal X}= \|\cdot \|_{L^2(0,T; H)} \vee \|\cdot \|_{C([0,T], H^{-\delta})}$. Recall that $\Phi^N(\cdot, \Phi_0)$ is the pathwise unique solution to \eqref{stoch-MHDEs} with initial condition $\Phi_0\in B_H(K)$, while $\Phi (\cdot, \Phi_0)$ is the unique global solution of the deterministic 3D MHD equations \eqref{thm-limit.1} with initial condition $\Phi_0$. Then we can find a subsequence of integers $\{N_i\}_{i\geq 1}$ and $\Phi^{N_i}_0\in B_H(K)$, such that (choose $\eps_0$ even smaller if necessary)
  \begin{equation}\label{proof-scaling-1}
  \P \Big(\|\Phi^{N_i}(\cdot, \Phi^{N_i}_0) -\Phi (\cdot, \Phi^{N_i}_0) \|_{\mathcal X} >\eps_0 \Big) \geq \eps_0 >0.
  \end{equation}
For each $i\geq 1$, let $Q^{N_i}$ be the law of $\Phi^{N_i}(\cdot, \Phi^{N_i}_0)$. Since $\{ \Phi^{N_i}_0 \}_{i\geq 1}$ is contained in the ball $B_H(K)$, there exists a subsequence of $\{\Phi^{N_i}_0 \}_{i\geq 1}$ (not relabelled for simplicity) converging weakly to some $\Phi_0\in B_H(K)$.

Similarly to the discussion at the beginning of this section, we can show that the family $\{Q^{N_i} \}_{i\geq 1}$ is tight on $\mathcal X= L^2(0,T; H) \cap C\big([0,T], H^{-\delta} \big)$, hence, up to a subsequence, $Q^{N_i}$ converges weakly to some probability measure $Q$ supported on $\mathcal X$. As a result, by Skorokhod's representation theorem, we can find a new probability space $\big( \tilde\Omega, \tilde{\mathcal F}, \tilde \P\big)$ and a sequence of processes $\big\{\tilde \Phi^{N_i} \big\}_{i\in \N}$ defined on $\tilde\Omega$, such that for each $i\in \N$, $\tilde \Phi^{N_i}$ has the same law $Q^{N_i}$ as $\Phi^{N_i}$, and $\tilde \P$-a.s., $\tilde \Phi^{N_i}$ converges as $i\to \infty$ to some $\tilde\Phi$ strongly in $\mathcal X$. As before, the limit $\tilde\Phi$ solves the deterministic 3D MHD equations \eqref{thm-limit.1} with initial condition $\Phi_0$. From this we conclude that $\tilde\Phi = \Phi (\cdot, \Phi_0)$, and thus, as $i\to \infty$, $\tilde \Phi^{N_i}$ converge in $\mathcal X$ to $\Phi (\cdot, \Phi_0)$ in probability, i.e., for any $\eps>0$,
  \begin{equation}\label{proof-scaling-2}
  \lim_{i\to \infty} \tilde\P \Big(\big\|\tilde \Phi^{N_i} -\Phi (\cdot, \Phi_0) \big\|_{\mathcal X} >\eps \Big) =0.
  \end{equation}
Note that $\tilde\Phi^{N_i} \stackrel{d}{\sim} Q^{N_i}$, \eqref{proof-scaling-1} implies
  \begin{equation}\label{proof-scaling-3}
  \tilde \P\Big(\big\|\tilde\Phi^{N_i} -\Phi (\cdot, \Phi^{N_i}_0) \big\|_{\mathcal X} >\eps_0 \Big) \geq \eps_0 >0.
  \end{equation}
We have the triangle inequality:
  \begin{equation}\label{proof-scaling-4}
  \big\|\tilde\Phi^{N_i} -\Phi (\cdot, \Phi^{N_i}_0) \big\|_{\mathcal X} \leq \big\|\tilde\Phi^{N_i} -\Phi(\cdot, \Phi_0) \big\|_{\mathcal X} + \big\|\Phi(\cdot, \Phi_0) -\Phi (\cdot, \Phi^{N_i}_0) \big\|_{\mathcal X}.
  \end{equation}
Recall that $\|\Phi^{N_i}_0 \|_{L^2} \leq K,\, i\geq 1$; Lemma \ref{lem-3D-MHD} implies that $\{\Phi(\cdot, \Phi^{N_i}_0) \}_{i\geq 1}$ is bounded in $L^\infty(0,T; H)$. This estimate and the first inequality in \eqref{lem-3D-MHD.2} further imply that the family is bounded in $L^2(0,T; V)$. One can also show its boundedness in $W^{1,2}\big(0,T; H^{-2} \big)$ by using the equations \eqref{thm-limit.1}. Then, by embedding results similar to those in \eqref{Simon-embedding}, the family $\big\{ \Phi (\cdot, \Phi^{N_i}_0) \big\}_{i\geq 1}$ is sequentially compact in $\mathcal X= L^2(0,T; H) \cap C\big([0,T], H^{-\delta} \big)$. Therefore, up to a subsequence, $\Phi (\cdot, \Phi^{N_i}_0)$ converges in $\mathcal X$ to some $\bar \Phi$, which can be shown to solve \eqref{thm-limit.1} since $\Phi^{N_i}_0$ converges weakly to $\Phi_0$. In other words, $\bar \Phi= \Phi(\cdot, \Phi_0)$ and $\big\|\Phi (\cdot, \Phi^{N_i}_0) -\Phi(\cdot, \Phi_0) \big\|_{\mathcal X} \to 0$ as $i\to \infty$. Combining this result with \eqref{proof-scaling-2}--\eqref{proof-scaling-4}, we get a contradiction.
\end{proof}

Finally we can provide the

\begin{proof}[Proof of Theorem \ref{thm-main}]
We divide the proof in 3 steps.

\emph{Step 1.} We fix $K>0$, and choose $\nu,\, R$ as in the proof of Theorem \ref{thm-limit}; we know that the unique global solution to \eqref{thm-limit.1} satisfies
  \begin{equation}\label{sect-1-bound}
  \|\Phi(\cdot\,, \Phi_0) \|_{C([0,T]; H^{-\delta})} \leq \|\Phi(\cdot\,, \Phi_0) \|_{L^\infty(0,T; L^2)}\leq C_K= R-1, \quad \mbox{for all } \Phi_0\in B_H(K).
  \end{equation}
Moreover, by the exponential decay \eqref{exponential-decay}, we can take $T>1$ big enough such that
  \begin{equation}\label{sect-1-bound.0}
  \|\Phi(\cdot\,, \Phi_0) \|_{L^2(T-1,T; H)} \leq r_0/2, \quad \mbox{for all } \Phi_0\in B_H(K),
  \end{equation}
where $r_0$ is the small number mentioned at the end of the paragraph involving the stochastic 3D MHD equations \eqref{stoch-MHD} (see \eqref{stoch-MHD.1} for the more precise form); the latter admit a unique global solution for any initial condition $\Phi_0\in B_H(r_0)$. Note that $r_0$ is independent of $\theta\in \ell^2$ and $\nu>0$. Without loss of generality we can assume $r_0\leq 1$.

\emph{Step 2.} Now we consider the approximating equations \eqref{stoch-MHDEs}, but with the same initial condition $\Phi_0$ as in \eqref{thm-limit.1}. Given $T>0$ as in \eqref{sect-1-bound.0} and arbitrary small $\eps>0$, Theorem \ref{thm-limit} implies that there exists $N_0= N_0(K, \nu, R, T, \eps)\in \N$ such that for all $N\geq N_0$, the pathwise unique strong solution $\Phi^N(\cdot, \Phi_0)$ of \eqref{stoch-MHDEs} satisfies, for all $\Phi_0\in B_H(K)$,
  \begin{equation}\label{sect-1-bound.1}
  \P \Big( \big\|\Phi^N(\cdot,\Phi_0) -\Phi(\cdot, \Phi_0) \big\|_{\mathcal X} \leq r_0/2 \Big) \geq 1- \eps,
  \end{equation}
where $\|\cdot \|_{\mathcal X}= \|\cdot \|_{L^2(0,T; H)} \vee \|\cdot \|_{C([0,T]; H^{-\delta})} $. In the sequel we fix such an $N\geq N_0$. Combining this with \eqref{sect-1-bound}, we deduce
  $$\P \Big( \big\|\Phi^N(\cdot,\Phi_0) \big\|_{C([0,T]; H^{-\delta})} < R \Big) \geq 1- \eps.$$
Defining the stopping times $\tau^N_R= \inf\big\{t>0: \|\Phi^N(t,\Phi_0) \|_{H^{-\delta}} >R\big\}$ ($\inf\emptyset =T$), then
  \begin{equation*}
   \P \big( \tau^N_R \geq T \big) \geq 1- \eps.
   \end{equation*}
For any $t\leq \tau^N_R$, we have
  $$f_R\big(\Phi^N(t,\Phi_0) \big)= f_R\big(\|\Phi^N(t,\Phi_0) \|_{H^{-\delta}} \big) =1,$$
and thus $\big\{\Phi^N(t,\Phi_0) \big\}_{t\leq \tau^N_R}$ is a solution to the following equation without cut-off:
  \begin{equation}\label{proof-thm-3D-NSE}
  \d \Phi^N + b\big(\Phi^N, \Phi^N\big)\,\d t = \big[ \Delta \Phi^N + S_{\theta^N} \big(\Phi^N \big) \big]\,\d t + \frac{C_\nu}{\|\theta^N \|_{\ell^2}} \sum_{k, \alpha} \theta^N_k \Pi\big( \sigma_{k,\alpha}\cdot \nabla \Phi^N \big) \,\d W^{k,\alpha}_t.
  \end{equation}
To sum up, with probability greater than $1-\eps$, uniformly in $\Phi_0\in B_H(K)$, this equation admits a pathwise unique solution on $[0,T]$.

\emph{Step 3.} Finally, let $\Omega_{N,\eps}$ be the event on the left hand side of \eqref{sect-1-bound.1}, then $\P(\Omega_{N,\eps}) \geq 1-\eps$. On the event $\Omega_{N,\eps}$, the triangle inequality yields
  \begin{equation*}
  \aligned
  \big\| \Phi^N(\cdot, \Phi_0) \big\|_{L^2(T-1,T; H)} &\leq \big\| \Phi^N(\cdot, \Phi_0) -\Phi(\cdot, \Phi_0) \big\|_{L^2(T-1,T; H)} + \| \Phi(\cdot, \Phi_0) \|_{L^2(T-1,T; H)}\\
  &\leq \frac{r_0}2 + \frac{r_0}2 =r_0,
  \endaligned
  \end{equation*}
where in the second step we have used \eqref{sect-1-bound.0}. This inequality holds for all $\omega\in \Omega_{N,\eps}$. As a result, for any $\omega\in \Omega_{N,\eps}$, there exists $t= t(\omega) \in [T-1, T]$ such that
  $$ \big\| \Phi^N(t(\omega), \Phi_0, \omega) \big\|_{L^2} \leq r_0 .$$
Therefore, restarting the equation \eqref{proof-thm-3D-NSE} at time $t(\omega)$  with the initial condition $\Phi^N(t(\omega), \Phi_0, \omega)$, we conclude that the solution extends to all $t> t(\omega)$ for every $\omega\in \Omega_{N,\eps}$. This completes the proof of Theorem \ref{thm-main} for $\nu>0$ taken as above and $\theta =\theta^N \in \ell^2$.
\end{proof}

\section{Appendix: a heuristic proof of \eqref{theta-N.2}}

In this part we provide a heuristic proof of the key limit \eqref{theta-N.2} in a special case; the full proof is quite long and the interested reader is referred to \cite[Section 5]{FlaLuo21}.

First, recall that, for a divergence free smooth vector field $v$,
  $$S_\theta(v)=\frac{C_\nu^2}{\|\theta \|_{\ell^2}^2} \sum_{k,\alpha} \theta_k^2\, \Pi\big[\sigma_{k,\alpha}\cdot\nabla \Pi(\sigma_{-k, \alpha} \cdot\nabla v) \big], $$
where $\Pi$ is the Leray projection operator. Let $\Pi^\perp$ be the operator which is orthogonal to $\Pi$, then we have
  $$S_\theta(v)= \frac{C_\nu^2}{\|\theta \|_{\ell^2}^2} \sum_{k,\alpha} \theta_k^2\, \Pi\big[\sigma_{k,\alpha}\cdot\nabla (\sigma_{-k, \alpha} \cdot\nabla v) \big] - \frac{C_\nu^2}{\|\theta \|_{\ell^2}^2} \sum_{k,\alpha} \theta_k^2\, \Pi\big[\sigma_{k,\alpha}\cdot\nabla \Pi^\perp (\sigma_{-k, \alpha} \cdot\nabla v) \big]. $$
It is not difficult to show that (cf. \cite[(2.4)]{FlaLuo21})
  $$\frac{C_\nu^2}{\|\theta \|_{\ell^2}^2} \sum_{k,\alpha} \theta_k^2\, \Pi\big[\sigma_{k,\alpha}\cdot\nabla (\sigma_{-k, \alpha} \cdot\nabla v) \big] =  \nu \Pi(\Delta v)= \nu  \Delta v; $$
thus, if we denote by
  $$S_\theta^\perp(v)= \frac{C_\nu^2}{\|\theta \|_{\ell^2}^2} \sum_{k,\alpha} \theta_k^2\, \Pi\big[\sigma_{k,\alpha}\cdot\nabla \Pi^\perp (\sigma_{-k, \alpha} \cdot\nabla v) \big], $$
then, it suffices to prove that, for $\theta^N$ defined in \eqref{theta-N},
  \begin{equation}\label{key-limit}
  \lim_{N\to \infty} S_{\theta^N}^\perp(v) = \frac25 \nu \Delta v \quad \mbox{holds in } L^2(\T^3,\R^3).
  \end{equation}
Below we will prove a weaker form of the limit in a particular case.

Recall that for a general vector field $X$, formally,
  \begin{equation*}
  \Pi^\perp X = \nabla \Delta^{-1} \div(X).
  \end{equation*}
On the other hand, if $X= \sum_{l\in \Z^3_0} X_l e_l$, $X_l\in \mathbb C^3$, then
  \begin{equation*}
  \Pi^\perp X= \sum_l \frac{l\cdot X_l}{|l|^2} l e_l = \nabla\bigg[ \frac1{2\pi {\rm i}} \sum_l \frac{l\cdot X_l}{|l|^2} e_l \bigg].
  \end{equation*}
We take a special vector field
  $$v= \sigma_{l,1} + \sigma_{l,2} = (a_{l,1} + a_{l,2})e_l, $$
where $a_{l,1}$ and $a_{l,2}$ are defined in Section \ref{subsec-notations}. Using any of the equalities above, one can prove (see \cite[Corollary 5.3]{FlaLuo21})
  $$\aligned
  S_{\theta^N}^\perp (v) &= -\frac{6\pi^2 \nu}{\|\theta^N \|_{\ell^2}^2} \sum_{\beta=1}^2 |l|^2 \Pi\bigg\{ \bigg[ \sum_{k} \big(\theta^N_k \big)^2 \sin^2(\angle_{k,l}) (a_{l,\beta}\cdot (k-l)) \frac{k-l}{|k-l|^2} \bigg] e_l \bigg\} \\
  &\sim -\frac{6\pi^2 \nu}{\|\theta^N \|_{\ell^2}^2} \sum_{\beta=1}^2 |l|^2 \Pi\bigg\{ \bigg[ \sum_{k} \big(\theta^N_k \big)^2 \sin^2(\angle_{k,l}) (a_{l,\beta}\cdot k) \frac{k}{|k|^2} \bigg] e_l \bigg\},
  \endaligned$$
where $\angle_{k,l}$ is the angle between the vectors $k$ and $l$, and $\sim$ means the difference between the two quantities vanishes as $N\to \infty$. The complex conjugate $\bar v$ of $v$ is divergence free, hence
  $$\aligned
  \big\< S_{\theta^N}^\perp (v), \bar v\big\>_{L^2} &\sim -\frac{6\pi^2 \nu}{\|\theta^N \|_{\ell^2}^2} \sum_{\beta=1}^2 |l|^2 \bigg\< \bigg[ \sum_{k} \big(\theta^N_k \big)^2 \sin^2(\angle_{k,l}) (a_{l,\beta}\cdot k) \frac{k}{|k|^2} \bigg] e_l, (a_{l,1} + a_{l,2}) e_{-l} \bigg\>_{L^2} \\
  &= -\frac{6\pi^2 \nu}{\|\theta^N \|_{\ell^2}^2} |l|^2 \sum_{\beta, \beta' =1}^2 \sum_{k} \big(\theta^N_k \big)^2 \sin^2(\angle_{k,l}) \frac{(a_{l,\beta}\cdot k) (a_{l,\beta'}\cdot k)}{|k|^2} .
  \endaligned $$
Recall that $\{a_{l,1}, a_{l,2} , \frac{l}{|l|} \}$ is an ONS of $\R^3$. By symmetry, the terms with $\beta \neq \beta'$ vanish, thus
  $$\aligned
  \big\< S_{\theta^N}^\perp (v), \bar v\big\>_{L^2} &\sim -\frac{6\pi^2 \nu}{\|\theta^N \|_{\ell^2}^2} |l|^2 \sum_{\beta =1}^2 \sum_{k} \big(\theta^N_k \big)^2 \sin^2(\angle_{k,l}) \frac{(a_{l,\beta}\cdot k)^2}{|k|^2}\\
  &= -\frac{6\pi^2 \nu}{\|\theta^N \|_{\ell^2}^2} |l|^2 \sum_{k} \big(\theta^N_k \big)^2 \sin^4(\angle_{k,l}),
  \endaligned $$
where we have used
  $$\sum_{\beta =1}^2 \frac{(a_{l,\beta}\cdot k)^2}{|k|^2} = 1- \bigg(\frac{k}{|k|} \cdot \frac{l}{|l|}\bigg)^2 = \sin^2(\angle_{k,l}). $$
Now, approximating the sums by integrals and changing to spherical variables yield
  $$\aligned \frac1{\|\theta^N \|_{\ell^2}^2} \sum_{k} \big(\theta^N_k \big)^2 \sin^4(\angle_{k,l}) &\sim \frac{\int_{\{ N\leq |x|\leq 2N\}} \frac{\sin^4(\angle_{x,l})}{|x|^{2\kappa}}\,\d x}{\int_{\{N\leq |x|\leq 2N\}} \frac{1}{|x|^{2\kappa}}\,\d x} = \frac{\int_N^{2N} \frac{\d r}{r^{2\kappa-2}} \int_0^\pi \sin^5 \psi\,\d\psi \int_0^{2\pi} \d\varphi}{\int_N^{2N} \frac{\d r}{r^{2\kappa-2}} \int_0^\pi \sin \psi\,\d\psi \int_0^{2\pi} \d\varphi} \\
  &= \frac12 \int_0^\pi \sin^5 \psi\,\d\psi = \frac8{15}.
  \endaligned $$
Thus, as $N\to \infty$,
  $$\big\< S_{\theta^N}^\perp (v), \bar v \big\>_{L^2} \to - 6\pi^2 \nu |l|^2 \cdot \frac8{15} = - \frac{16}5 \pi^2 \nu |l|^2 = \frac25 \nu \<\Delta v, \bar v\>_{L^2}, $$
since $\Delta v= -4\pi^2|l|^2 v = -4\pi^2|l|^2 ( \sigma_{l,1} + \sigma_{l,2})$.

\bigskip

\noindent \textbf{Acknowledgements.} The author is grateful to the financial supports of the National Key R\&D Program of China (No. 2020YFA0712700) and the National Natural Science Foundation of China (Nos. 11688101, 11931004, 12090014).

\end{document}